\documentclass{tran-l}

 \newtheorem{thm}{Theorem}[subsection]

 \newtheorem{exam}[thm]{Example}
 \newtheorem{prop}[thm]{Proposition}
 \theoremstyle{definition}
 \newtheorem{defn}[thm]{Definition}
 \theoremstyle{observation}
 
 \theoremstyle{remark}
 \newtheorem{rem}[thm]{Remark}
 \numberwithin{equation}{subsection}
  \usepackage{graphicx}
 \usepackage{placeins}
 \usepackage{amssymb}
 \usepackage{amsmath}
 \usepackage[centertags]{amsmath}
\usepackage{amsfonts}
\begin{document}
%\title{Rough Sets in Graphs Using Similarity Relations]
% {Rough Sets in Graphs Using Similarity Relations}
%\author{Imran Javaid*, Shahroz Ali, Shahid Ur Rehman}
%%\subjclass{Primary: , Secondary: }
%\subjclass[2010]{Primary 05C30, 05A18; Secondary 68R10.}
%\keywords{Rough Sets, Automorphisms, Orbit, Core, Reduct,
%Discernibility Matrix}. \\
%%\indent 2010 {\it Mathematics Subject Classification.} 05C50\\
%%\subjclass[2010]{Primary 05C30, 05A18; Secondary 68R10.}
%%\indent $^*$ Corresponding author: imran.javaid@bzu.edu.pk}
%%\indent This research of the authors was partially supported by the
%%Higher Education Commission of\\
%%\indent Pakistan}
%\address{Centre for advanced studies in Pure and Applied Mathematics,
%Bahauddin Zakariya University Multan, Pakistan\newline Email:
%imran.javaid@bzu.edu.pk, shahroz.ali0304@gmail.com, shahidurrehman1982@gmail.com}

\title[Rough Sets in Graphs Using Similarity Relations]
 {Rough Sets in Graphs Using Similarity Relations}

\author{Imran Javaid*, Shahroz Ali, Shahid Ur Rehman, Aqsa Shah}

\address{Centre for Advanced Studies in Pure and Applied Mathematics, Bahauddin Zakariya University, Multan,
Pakistan.}

\email{imran.javaid@bzu.edu.pk, shahroz.ali0304@gmail.com, shahidurrehman1982@gmail.com, aqsa6977@gmail.com}

%\thanks{This work was completed with the support of an Izaak Walton Killam Memorial Scholarship.}

%\thanks{The author was also supported in part by the Research Council of Slovenia.}

%\subjclass{Primary ; Secondary }
%\subjclass{Primary 47A15; Secondary 46A32, 47D20}
%\subjclass[2010]{Primary 05C30, 05A18; Secondary 68R10.}
%\subjclass{Primary 05C30; 05A18; Secondary 68R10; 68R10}
%\indent 2020 {\it Mathematics Subject Classification.}{Primary 05C30; 05A18; Secondary 68R10; 68R10}\\
. 
\thanks{$^*$ Corresponding author: imran.javaid@bzu.edu.pk}

%\date{February 15, 1995 and, in revised form, July 6, 1995.}

\dedicatory{}

%\commby{Daniel J. Rudolph}

%%% ----------------------------------------------------------------------

\begin{abstract}
In this paper, we use theory of rough set to study graphs using the
concept of orbits. We investigate the indiscernibility partitions
and approximations of graphs induced by orbits of graphs. We also study rough membership functions, essential sets,
discernibility matrix and their relationships for graphs.
%In particular, we interpret the information systems of graphs by linking the relationship of discernibility matrix and extended core graphs thoroughly.
\end{abstract}

%%% ----------------------------------------------------------------------
\maketitle
%%% ----------------------------------------------------------------------

\section*{Introduction}
Rough set theory (RST), introduced by Pawlak \cite{16}, provides
elegant and powerful techniques to extract  information from data
associated with various structures. It also provides terminology to
minimize the number of significant attributes in data tables, also
called information systems. It is always interesting and insightful
to know about the objects having similar characteristics to simplify
the study of objects under consideration. The objects which are
indistinguishable from each other are called information granules.
Zadeh introduced and studied information granularity in \cite{29}.
Information granules are objects placed together due to similarity
of features of interest and are dealt together yielding partition of
objects based upon features. Many applications of granular computing
appear in connection with fuzzy set theory \cite{11,20}, information
science \cite{14,28}, data mining \cite{10,13,27}, formal concept
analysis \cite{12,27}, database theory \cite{9,21} and rough set
theory \cite{2,17,18,19,24} etc. in fields, such as medicine,
economy, finance, business, environment, electrical and computer
engineering. Granular computing and rough sets have been studied in
context of graphs \cite{5,6}, digraph \cite{15} and hyper graphs
\cite{3}. The reader is referred to the following \cite{4,22,23,25}
for further reading in this area of study. Chiaselloti et al.
\cite{5,6,7}  have studied well known families of graphs using the
notion of neighbourhood of vertices. They have remarked two vertices
$x, \, y$ as indiscernible with respect to a vertex subset
$\mathcal{A}$ when these vertices have same neighbours with respect
to $\mathcal{A}$. Symbolically,
$x\equiv_{\mathcal{A}}y\Leftrightarrow N(x)\cap \mathcal{A}=N(y)\cap
\mathcal{A}$. In \cite{7}, they have remarked $\mathcal{A}$ as
symmetry axis and two vertices are indiscernible with respect to
$\mathcal{A}$ if they are in a `symmetrical' position with respect
to all vertices of $\mathcal{A}$. Equivalence classes induced by
indiscernibility relation of $\mathcal{A}$ are the granules and
yield partition of the vertex set of the graph, denoted by
$\gamma_{\mathcal{A}}(\mathcal{G})$. Rough set theory can be thought
of as a particular type of granular computing. Chiaselloti et al.
\cite{7} referred the triplet
$(\mathcal{G},\mathcal{A},\gamma_{\mathcal{A}}(\mathcal{G}))$ as
$\mathcal{A}$-granular reference system. Intersection of granular
computing and rough set theory provides tools to analyse graphs,
digraphs and hypergraphs. The information about granules will
simplify studies associated to graphs and can give important
insights about graphs, enriching their applications in various
fields and introduce new perspectives.

By associating rough set and graphs, properties of rough sets can be
understood with the help of graphs and those of graphs can be
explained by using RST. Graphs have been studied using rough set
\cite{4,6,15} and rough sets may yield graphs \cite{23}. A strategy
developed in one domain can give insights in another domain. It is
worth noting that the partition in the granular reference system may
be induced using different paradigms and it will give useful
insights in various directions. Making comparisons, identifying
similarities and studying differences are canonical approaches of
study. In this paper, we study graphs using the idea of orbit of
vertices and we remark two vertices $u,v$ as indiscernible when
orbits intersect with $\mathcal{A}$ in the same set or equivalently
two vertices $u,v$ as indiscernible if they belong to same orbits of
elements of $\mathcal{A}$. Chiaselotti et al. \cite{5} remarked two
vertices as indiscernible if their open neighbourhoods intersect
with a given set $\mathcal{A}$ in the same set. We consider a
different approach from that of Chiaselotti et al. \cite{5} and
consider similarity relation yielded by graph automorphisms to fill
the entries of information table. We will study indiscernibility
relations introduced with the help of orbits and prove several
results including parallel to those of Chaiselotti et al. \cite{5,
6, 7} using the concept of orbits of vertices. We present
introductory terminology needed for the paper first.

A graph $\mathcal{G}$ is an ordered pair consists on two finite sets
$\mathcal{V}(\mathcal{G})$ and $\mathcal{E}(\mathcal{G})$ called
vertex set and edge set respectively. In this paper, all considered
graphs are simple and non-trivial, represented by
$\mathcal{G}=(\mathcal{V}(\mathcal{G}),\mathcal{E}(\mathcal{G}))$,
when there is no doubt, we simply write
$\mathcal{G}=(\mathcal{V},\mathcal{E})$. Two vertices are called
adjacent or neighbours of each other if there is an edge between
them. The cardinality of vertex set $\mathcal{V}$ and edge set
$\mathcal{E}$ of $\mathcal{G}$ is termed as $order$ and $size$ of
$\mathcal{G}$ respectively. For $x,y\in\mathcal{V}$, $x\sim y$ means
$x,y$ are adjacent to each other by an edge and $x \nsim y$ means
$x,y$ are non-adjacent. $\textit{Open neighbour}$ of
$x\in\mathcal{V}$ is defined as
$\mathcal{N}_{\mathcal{G}}(x)=\{y\in\mathcal{V}: x\sim y\}$.
Similarly, $\textit{closed neighbour}$ of $x$ is defined as
$\mathcal{N}_{\mathcal{G}}[x]=\mathcal{N}_{\mathcal{G}}(x)\cup
\{x\}$.

A $permutation$ of a set is a bijection from the set to itself. A
graph $automorphism$ is a permutation of vertex set that preserves
adjacency and non-adjacency of the vertices. Alternatively,
$\pi:\mathcal{V}(\mathcal{G})\rightarrow \mathcal{V}(\mathcal{G})$
is an automorphism of a graph $\mathcal{G}$ if for all
$x,y\in\mathcal{V}(\mathcal{G}),\pi(x)\sim \pi(y)$ if and only if
$x\sim y$. The collection of all automorphisms of graph
$\mathcal{G}$ forms a group, named as the $\textit{automorphism
group}$ of the graph $\mathcal{G}$. We use $\Gamma(\mathcal{G})$ or
$\Gamma$ if $\mathcal{G}$ is clear from context, to represent the
group of automorphisms of graph $\mathcal{G}$. For
$y\in\mathcal{V}$, the $orbit$ of $y$, denoted by $\mathcal{O}(y)$
is defined as $\mathcal{O}(y)=\{\pi(y):\pi\in\Gamma\}$ and for any
set $\mathcal{A}$, the orbit of $\mathcal{A}$ is represented by
$\mathcal{O}(\mathcal{A})$ is define as
$\mathcal{O}(\mathcal{A})=\cup\mathcal{O}(x_i)\forall
x_i\in\mathcal{A}$. Two vertices in same orbit are called $similar$
vertices.

A partition $\gamma$ on a finite set $\mathcal{V}$ is a family of
non-empty subsets $B_1,B_2,\dots, B_n$ of $\mathcal{V}$ such that
$B_i\cap B_j=\emptyset$ for all $i\neq j$ and
$\cup_{i=1}^nB_i=\mathcal{V}$. The subsets $B_1,B_2,\dots, B_n$ are
termed as blocks of $\gamma$, we write $\gamma:=B_1|B_2|\dots|B_n$
to represent that $\gamma$ is a partition of sets with blocks
$B_1,B_2,\dots, B_n$. If $x \in \mathcal{V}$, we denote $\gamma (x)$
as the block of $\gamma$ containing the element $x$. An $\textit{information table}$ is denoted by
$\mathcal{I}=\langle \mathcal{U}, \mathbb{A},
\mathcal{V}al,\mathcal{F} \rangle$ is a quadruple, where
$\mathcal{U}$ represents the universal set of objects, $\mathbb{A}$
represents the attribute set, $\mathcal{V}al$ is a set of outcomes
and $\mathcal{F}:\mathcal{U}\times \mathbb{A}\rightarrow
\mathcal{V}al$ represents an information map. The information table
is called $\textit{Boolean}$ if $\mathcal{V}al=\{0,1\}$. For the purpose of our paper, both universal set $\mathcal{U}$ and attribute set $\mathbb{A}$ are the vertex set $\mathcal{V}$ of the graph $\mathcal{G}$ and the
$indescernibility$ relation $\equiv_{\mathcal{A}}$ between the
two vertices of the graphs is an equivalence relation which
depends upon a set of attributes $\mathcal{A}\subseteq \mathcal{V}$ defined as: $x\equiv_{\mathcal{A}} y$ if and only if for all $z\in \mathcal{A}$,
if $z\in\mathcal{O}(x)$ then $z\in\mathcal{O}(y)$. If there exists $z\in \mathcal{A}$ such that $z$ does not belong to $\mathcal{O}(x)$ and $\mathcal{O}(y)$ simultaneously, then we say $x$ and $y$ do not belong the same equivalence class and write it as $\urcorner(x\equiv_{\mathcal{A}}y)$. Equivalence
classes also called $\mathcal{A}$-granules of object $x$ under
consideration of $\equiv_{\mathcal{A}}$ are represented by
$\mathcal{O}_{\mathcal{A}}(x)$ and the indiscernibility partition of
$\mathcal{I}$ with respect to $\mathcal{A}$ is the family of all
equivalence classes, $i.e.,$
$\gamma_{\mathcal{A}}(\mathcal{G}):=\{\mathcal{O}_{\mathcal{A}}(x):x\in
\mathcal{V}\}$. Given an attribute subset $\mathcal{A}$ and object
subset $\mathcal{Q}$, lower and upper approximations of
$\mathcal{Q}$ by considering the information presented by
$\mathcal{A}$ are:
\begin{center}
$\mathbb{L}_{\mathcal{A}}(\mathcal{Q})=\{x\in\mathcal{V}:\mathcal{O}_{\mathcal{A}}(x)\subseteq
\mathcal{Q}\}=\cup\{\mathcal{D}\in\gamma_{\mathcal{A}}(\mathcal{G}):\mathcal{D}\subseteq \mathcal{Q}\}$\\
$\mathbb{U}_{\mathcal{A}}(\mathcal{Q})=\{x\in\mathcal{V}:\mathcal{O}_{\mathcal{A}}(x)\cap
\mathcal{Q}\neq\emptyset\}=\cup\{\mathcal{D}\in\gamma_{\mathcal{A}}(\mathcal{G}):\mathcal{D}\cap
\mathcal{Q}\neq\emptyset\}$
\end{center}
In general, $\mathbb{L}_{\mathcal{A}}(\mathcal{Q})$ is subset of
$\mathcal{Q}$ while $\mathbb{U}_{\mathcal{A}}(\mathcal{Q})$ is a
superset of $\mathcal{Q}$. A subset $\mathcal{Q}$ is termed as
$\mathcal{A}\textit{-exact}$ if and only if
$\mathbb{L}_{\mathcal{A}}(\mathcal{Q})=\mathbb{U}_{\mathcal{A}}(\mathcal{Q})$,
$\mathcal{A}\textit{-rough}$ otherwise. Let
$\mathcal{A},\mathcal{Q}\subseteq\mathcal{V}$ then rough membership
function is defined by:
\begin{center}
$\mu_{\mathcal{Q}}^{\mathcal{A}}(x)=\frac{|\mathcal{O}_{\mathcal{A}}(x)\cap
\mathcal{Q}|}{|\mathcal{O}_{\mathcal{A}}(x)|}=\frac{|[x]_{\mathcal{A}}\cap
\mathcal{Q}|}{|[x]_{\mathcal{A}}|}.$
\end{center}
For $\mathcal{A},\mathcal{D}\subseteq\mathcal{V}$ the positive region
is defined as
$POS_{\mathcal{A}}(\mathcal{D})=\{x\in\mathcal{V}:\mathcal{O}_{\mathcal{A}}(x)
\subseteq\mathcal{O}_{\mathcal{D}}(x)\}$, and the number
$deg_{\mathcal{A}}({\mathcal{D}})=\frac{|POS_{\mathcal{A}}({\mathcal{D}})|}{|\mathcal{V}|}$
is called $\mathcal{A}$-degree dependency of $\mathcal{D}$.

Important terms associated with rough set theory include the indiscernibility relations,
lower and upper approximations, reducts, essential sets  and the discernibility matrix.  The reduct \cite{16} of an
information table consists of the set of attributes which provide
same information and characterization as does the set of all attributes. Essential sets and the core \cite{16} of an information table usually consists of the set of attributes which are considered most important in
characterization of an information table. The discernibility matrix
\cite{22} is a square matrix with $ij^{th}$ entry consists of
attributes for which the attribute value is different for objects
$x_i$ and $x_j$.

%\textbf{Equitable partition}
%\par This paper is organized as follows: In section 1, we introduce
%an information system in form of table for graphs and investigate
%the equivalence classes of graph. We study the indiscernible
%partitions of graphs in terms of orbits. We state our results
%regarding the approximations of rough sets of graphs like; lower and
%upper approximation set, rough membership, dependency measures and
%also also investigate properties of rough sets of some special cases
%for orbits. In section 2, we discuss the action of automorphisms in
%connection with rough approximations like; indiscernibility
%partitions, reducts and core. In section 3, we discuss the extension
%of core of graphs and their reducts. In the last concluding remarks
%are given.

\par This paper is organized as follows: In section 1, we introduce and study the
indiscernibility partitions and we also discuss the action of
automorphisms in connection with indiscernibility relations. In
section 2, we study lower and upper approximations of graphs in
terms of orbits. We also examine the rough membership function and
dependency measures for graphs. In section 3, we discuss essential
sets, discernibility matrix and their relationships. Conclusions are given as last section of the paper.

%%% ---------------------------------------------------------------------

%\section{Approximation space using orbits}
%For a given graph $\mathcal{G}$, we set $U = \mathcal{V}$ and $R
%=\sim_{s}$ so the approximation space is $(\mathcal{V}, \sim_{s})$.
%Clearly, $R(x) = \mathcal{O}(x)$ for each $x \in \mathcal{V}$ and
%using the fact that for any two distinct vertices $x, y\in V $
%either $\mathcal{O}(x)= \mathcal{O}(y)$ or $\mathcal{O}(x)\neq
%\mathcal{O}(y)$, we have the following straight forward result.
%
%\begin{prop}
%A set $X \in \mathcal{V(\mathcal{G})}$ is exact if and only if $ X =
%\cup _{x\in X}\mathcal{O}(x)$.
%\end{prop}
%
%\begin{proof}
%To prove that $X$ is exact, we need to prove that $BNR(X) = R^{*}
%(X) \backslash R^{*}(X)$ is empty set. Since for $v\in
%\mathcal{V(\mathcal{G})}$ if $\mathcal{O}(v) \cap X\neq \varphi$.
%\end{proof}
%
%\begin{prop}
%If $\mathcal{G}$ is a graph such that $\mathcal{O}(x) =
%\mathcal{V(\mathcal{G})}$ then every proper subset $X$ of
%$\mathcal{V(\mathcal{G})}$ is rough.
%\end{prop}
%
%\begin{prop}
%If $\mathcal{G}$ is a graph such that $\mathcal{O}(x) = \{ x \}$ for
%each $x \in \mathcal{V(\mathcal{G})}$ then every proper subset $X$
%of $\mathcal{V(\mathcal{G})}$ is exact.
%\end{prop}
%
%\begin{prop}
%If $x, y \in \mathcal{V(\mathcal{G})}$ such that $O(x)\neq O(y)$ in
%$\mathcal{G}$ and $O(x)\cup O(y) = \mathcal{V(\mathcal{G})}$, then a
%subset X of $\mathcal{V(\mathcal{G})}$ is rough if and only if $X
%\neq \mathcal{O}(x), \, X \neq \mathcal{O}(y)$ and $X \neq
%\mathcal{V(\mathcal{G})}$.
%\end{prop}
%

\section{Indiscernibility Partitions of Graphs}
An undirected simple graph $\mathcal{G}$ is described by an
information table represented by $\mathcal{I}(\mathcal{G})$. Suppose
that the vertex set $\mathcal{V}(\mathcal{G})=\{x_1,x_2,\dots,x_n\}$
and both universal and attribute sets of information table
$\mathcal{I}(\mathcal{G})$ are equal to $\mathcal{V}$ and
characterize the information map as follows: for $x_i \in
\mathcal{A}$, $\mathcal{F}(x_i,x_j)=1$ if $x_j\in \mathcal{O}(x_i)$
and $\mathcal{F}(x_i,x_j)=0$ if $x_j\notin \mathcal{O}(x_i)$.

In theorem \ref{4}, we show how indiscernibility relation
$\equiv_{\mathcal{A}}$ and the concept of orbit are related.

\begin{thm}\label{4}
For an attribute subset $\mathcal{A}\subseteq \mathcal{V}$, any two
vertices $x,y\in \mathcal{V}$ are indiscernible if and only if
$\mathcal{O}(x)\cap \mathcal{A}=\mathcal{O}(y)\cap \mathcal{A}$.
\end{thm}
\begin{proof}
If $z\in\mathcal{O}(x)\cap\mathcal{A}$ then
$\mathcal{F}(x,z)=1=\mathcal{F}(y,z)$ by definition it implies that
$z\in\mathcal{O}(y)\cap\mathcal{A}$. Hence $\mathcal{O}(x)\cap
\mathcal{A}\subseteq\mathcal{O}(y)\cap \mathcal{A}$. Using similar
argument, $\mathcal{O}(y)\cap \mathcal{A}\subseteq\mathcal{O}(x)\cap
\mathcal{A}$. Hence $\mathcal{O}(x)\cap
\mathcal{A}=\mathcal{O}(y)\cap
\mathcal{A}$.\\
Now suppose $\mathcal{O}(x)\cap \mathcal{A}=\mathcal{O}(y)\cap
\mathcal{A}$, then $z\in\mathcal{O}(x)\cap \mathcal{A}$ implies that
$\mathcal{F}(x,z)=1$ which implies that $\mathcal{F}(y,z)=1$. Hence
$\mathcal{F}(x,z)=\mathcal{F}(y,z)$, which implies that
$x\equiv_{\mathcal{A}}y$.
\end{proof}
Before, we provide a sketch of indiscernibility partitions of a
graph $\mathcal{G}$ in the form of orbits, please note the
relationship of similarity of vertices of graph with
indiscernibility of vertices. If $\mathcal{O}(x)=\mathcal{O}(y)$ in
$\mathcal{G}$ then $x\equiv_\mathcal{A} y$ for any attribute subset
$\mathcal{A}$. Moreover, converse is not true in general. Consider a
path graph $P_5$ with the vertex set $\{x_1,x_2,x_3,x_4,x_5\}$ and
$x_i$ is adjacent to $x_{i+1}$ for $i=1,2,3,4$. The orbits of $P_5$
are $\{x_1,x_5\}$, $\{x_2,x_4\}$ and $\{x_3\}$. For
$\mathcal{A}=\{x_1\}$ be an attribute set the
$\mathcal{A}$-equivalence classes are $\{x_1,x_5\}$,
$\{x_2,x_3,x_4\}$. Check that $x_2\equiv_\mathcal{A} x_3$ but $x_2$
is not
similar to $x_3$ in $P_5$, hence converse is not true in general.\\

For a subset $\mathcal{A}$ of $\mathcal{V}$, set
$\mathcal{O}(\mathcal{A})=\cup_{x\in \mathcal{A}}\mathcal{O}(x)$. In
the following proposition, we give a complete sketch for the
indiscernibility partitions of graph $\mathcal{G}$ in the form of
orbits.

\begin{prop}\label{22}
    Let $\mathcal{G}$ be a graph and $\mathcal{A}\subseteq \mathcal{V}$ be a given attribute subset.
    If $\mathcal{B}_\mathcal{A}(\mathcal{G})=(\mathcal{O}(\mathcal{A}))^c$ and
    $\mathcal{A}=\{x_1,x_2,\ldots,x_k\}$ such that $\mathcal{O}(x_i)\cap \mathcal{O}(x_j)=\emptyset$
    for each $i\neq j$ and $\mathcal{O}(\mathcal{A})=\cup^{k}_{i=1}\mathcal{O}(x_i)$, then
    $\gamma_{\mathcal{A}}(\mathcal{G})=\mathcal{B}_\mathcal{A}(\mathcal{G})|
    \mathcal{O}(x_1)|\mathcal{O}(x_2)|
    \ldots|\mathcal{O}(x_k)$.
\end{prop}
\begin{proof}Let $x,y\in \mathcal{V}$ such that $x\equiv_\mathcal{A} y$ which gives that
$\mathcal{O}(x)\cap \mathcal{A}=\mathcal{O}(y)\cap \mathcal{A}$ and we have the following two cases.\\
    Case 1. Suppose $\mathcal{O}(x)\cap \mathcal{A}=\mathcal{O}(y)\cap \mathcal{A}=\emptyset$,
     clearly $x,y\in \mathcal{B}_\mathcal{A}(\mathcal{G})$ which gives that $\mathcal{B}_\mathcal{A}(\mathcal{G})$
     is an $\mathcal{A}$-equivalence class in $\mathcal{V}$.\\
    Case 2. Suppose $\mathcal{O}(x)\cap \mathcal{A}=\mathcal{O}(y)\cap \mathcal{A}\neq\emptyset$
    implies that $x,y\in \mathcal{O}(\mathcal{A})=\cup^{k}_{i=1}\mathcal{O}(x_i)$ and
    $\mathcal{O}(x_i)\cap \mathcal{O}(x_j)=\emptyset$ gives that $x,y\in \mathcal{O}(x_i)$ for some
    $1\leq i\leq k$ because if $x\in \mathcal{O}(x_i)$ and $y\in \mathcal{O}(x_j)$ for $i\neq j$ then
    $\mathcal{O}(x)\cap \mathcal{A}\neq\mathcal{O}(y)\cap \mathcal{A}$. Note that $\mathcal{O}(x)= \mathcal{O}(y)$
    implies $x\equiv_\mathcal{A} y$, hence
     $\mathcal{O}_\mathcal{A}(x)=\mathcal{O}_\mathcal{A}(y)=\mathcal{O}_\mathcal{A}(x_i)
     =\mathcal{O}(x_i)$
     and $\mathcal{O}_\mathcal{A}(x_i)=\mathcal{O}(x_i)$ for each $i$. \\
    Concluding above two cases and using the fact that $\mathcal{B}_\mathcal{A}(\mathcal{G})$,
    $\mathcal{O}(x_1)$, $\mathcal{O}(x_2),\ldots,\mathcal{O}(x_k)$ gives a partition of $\mathcal{V}$,
    we have $\gamma_{\mathcal{A}}(\mathcal{G})=\mathcal{B}_\mathcal{A}(\mathcal{G})|\mathcal{O}(x_1)|
    \mathcal{O}(x_2)|\ldots|\mathcal{O}(x_k)$.
\end{proof}

\begin{exam}
Consider a graph $\mathcal{G}$ in  FIGURE 1. %and its information
%table in TABLE 1.
\begin{figure}[h]
\begin{center}
  \includegraphics[width=5.5cm]{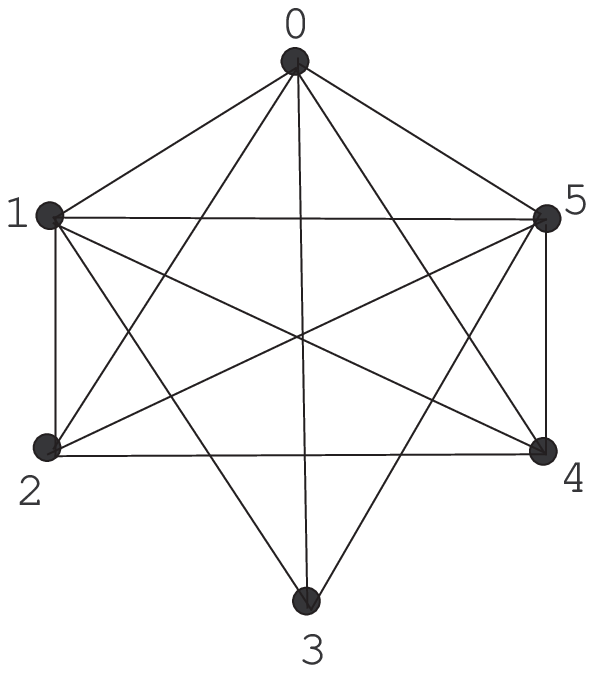}
  \caption{A Graph $\mathcal{G}$}\label{1}
\end{center}
\end{figure}
%\FloatBarrier
%\begin{table}[h]
%\centering
%\begin{tabular}{|c|c|c|c|c|c|c|}
%\hline
%    & 0 & 1 & 2 & 3 & 4 & 5 \\
%    \hline
%  0 & 1 & 1 & 0 & 0 & 0 & 1 \\
%  1 & 1 & 1 & 0 & 0 & 0 & 1 \\
%  2 & 0 & 0 & 1 & 0 & 1 & 0 \\
%  3 & 0 & 0 & 0 & 1 & 0 & 0 \\
%  4 & 0 & 0 & 1 & 0 & 1 & 0 \\
%  5 & 1 & 1 & 0 & 0 & 0 & 1 \\
%  \hline
%\end{tabular}
%\vspace*{10mm}
%
%\caption{Information Table of The Graph of Figure 1}
%\end{table}
\FloatBarrier In FIGURE 1, subscripts of the vertices are used as
labels. We have
$\mathcal{O}(0)=\{0,5,1\}=\mathcal{O}(5)=\mathcal{O}(1),
\mathcal{O}(4)=\{2,4\}=\mathcal{O}(2),\mathcal{O}(3)=\{3\}$. Fix a
vertex subset $\mathcal{A}=\{1,5\}$, then we have
$0\equiv_{\mathcal{A}} 1 \equiv_{\mathcal{A}} 5$, and similarly
$2\equiv_{\mathcal{A}} 3 \equiv_{\mathcal{A}} 4$ so
$\gamma_{\mathcal{A}}(\mathcal{G})=015|234$.
\end{exam}

%\textbf{It is interesting question to identify conditions under which for given two attribute sets
It is an interesting question to identify conditions under which for
given two attribute sets $\mathcal{A}_1$ and $\mathcal{A}_2$,
$\gamma_{\mathcal{A}_1}(\mathcal{G})=\gamma_{\mathcal{A}_2}(\mathcal{G})$.
We now provide a necessary and sufficient condition under which the
indiscernible partitions associated to two different attribute sets
are same.

\begin{prop}\label{15}
For an attribute subset $\mathcal{A}\subseteq\mathbb{A}$, define
$\mathcal{B}_{\mathcal{A}}(\mathcal{G})=(\mathcal{O}(\mathcal{A}))^c$.
Let $\mathcal{A}_1,\mathcal{A}_2$ be two attribute subsets, then
$\gamma_{\mathcal{A}_1}(\mathcal{G})=\gamma_{\mathcal{A}_2}(\mathcal{G})$
if and only if one of the following hold;
\begin{enumerate}
  \item $\mathcal{B}_{\mathcal{A}_1}(\mathcal{G})=\mathcal{B}_{\mathcal{A}_2}(\mathcal{G})$
  \item $\mathcal{B}_{\mathcal{A}_1}(\mathcal{G}),\mathcal{B}_{\mathcal{A}_2}(\mathcal{G})$ are either
  empty or consists of only one orbit.
\end{enumerate}
\end{prop}
\begin{proof}
(1) Suppose that
$\mathcal{B}_{\mathcal{A}_1}(\mathcal{G})=\mathcal{B}_{\mathcal{A}_2}(\mathcal{G})$
then $\mathcal{O}(\mathcal{A}_1)=\mathcal{O}(\mathcal{A}_2)$. Hence
by Proposition \ref{22},
$\gamma_{\mathcal{A}_1}(\mathcal{G})=\gamma_{\mathcal{A}_2}(\mathcal{G})$.\\
(2) Suppose that
$\mathcal{B}_{\mathcal{A}_1}(\mathcal{G})=\emptyset$ and
$\mathcal{B}_{\mathcal{A}_2}(\mathcal{G})$ consists of only one
orbit. Then the orbit of $\mathcal{B}_{\mathcal{A}_2}(\mathcal{G})$
must coincide with at least one orbit of
$\gamma_{\mathcal{A}_1}(\mathcal{G})$. Hence by Proposition
\ref{22},
$\gamma_{\mathcal{A}_1}(\mathcal{G})=\gamma_{\mathcal{A}_2}(\mathcal{G})$.
Similarly, if
 $\mathcal{B}_{\mathcal{A}_2}(\mathcal{G})=\emptyset$ and $\mathcal{B}_{\mathcal{A}_1}(\mathcal{G})$
 consists of only one orbit, then $\gamma_{\mathcal{A}_1}(\mathcal{G})=\gamma_{\mathcal{A}_2}(\mathcal{G})$.\\
Now, suppose contrary that
$\mathcal{B}_{\mathcal{A}_1}(\mathcal{G})\neq\mathcal{B}_{\mathcal{A}_2}(\mathcal{G})$
and neither $\mathcal{B}_{\mathcal{A}_1}(\mathcal{G})$ nor
$\mathcal{B}_{\mathcal{A}_2}(\mathcal{G})$ is empty or consists of
only one orbit. By proposition \ref{22}, this implies that
$\gamma_{\mathcal{A}_1}(\mathcal{G})\neq\gamma_{\mathcal{A}_2}(\mathcal{G})$
but
$\gamma_{\mathcal{A}_1}(\mathcal{G})=\gamma_{\mathcal{A}_2}(\mathcal{G})$,
a contradiction.
\end{proof}

Suppose that a graph $\mathcal{G}$ have $n$ orbits then following
propositions are obvious.

\begin{prop}\label{16}
If a graph $\mathcal{G}$ has $n$ orbits and attribute subsets
$\mathcal{A}_1$ and $\mathcal{A}_2$ contain elements of $n$ or $n-1$
orbits then
$\gamma_{\mathcal{A}_1}(\mathcal{G})=\gamma_{\mathcal{A}_2}(\mathcal{G})$.
\end{prop}
\begin{proof}
  Since graph $\mathcal{G}$ has $n$ orbits. If $\mathcal{O}(\mathcal{A}_1)=\mathcal{O}(\mathcal{A}_2)$,
  then $\mathcal{B}_{\mathcal{A}_1}(\mathcal{G})=\mathcal{B}_{\mathcal{A}_2}(\mathcal{G})$
  hence by proposition \ref{15}, $\gamma_{\mathcal{A}_1}(\mathcal{G})=\gamma_{\mathcal{A}_2}(\mathcal{G})$.
  Now suppose that $\mathcal{A}_1$ contains elements of $n$ orbits and $\mathcal{A}_2$
  contains elements of $n-1$ orbits, then by second part of proposition \ref{15},
  $\gamma_{\mathcal{A}_1}(\mathcal{G})=\gamma_{\mathcal{A}_2}(\mathcal{G})$.
\end{proof}

\begin{prop}
If $\mathcal{A}_1$ and $\mathcal{A}_2$ are two subsets of attributes
with elements from at most $n-2$ orbits then
$\gamma_{\mathcal{A}_1}(\mathcal{G})=\gamma_{\mathcal{A}_2}(\mathcal{G})$
if and only if
$\mathcal{O}(\mathcal{A}_1)=\mathcal{O}(\mathcal{A}_2)$.
\end{prop}
\begin{proof}
  As $\mathcal{O}(\mathcal{A}_1)=\mathcal{O}(\mathcal{A}_2)$, by proposition \ref{15},
  $\gamma_{\mathcal{A}_1}(\mathcal{G})=\gamma_{\mathcal{A}_2}(\mathcal{G})$.
  Suppose contrary that $\mathcal{O}(\mathcal{A}_1)\neq\mathcal{O}(\mathcal{A}_2)$
  but $\gamma_{\mathcal{A}_1}(\mathcal{G})=\gamma_{\mathcal{A}_2}(\mathcal{G})$.
  Since $\mathcal{A}_1$ and $\mathcal{A}_2$ contain elements of at most $n-2$ orbits.
  Hence their complement will consist of the elements of at least two different orbits
  which implies that $\mathcal{B}_{\mathcal{A}_1}(\mathcal{G})\neq\mathcal{B}_{\mathcal{A}_2}(\mathcal{G})$,
  yielding $\gamma_{\mathcal{A}_1}(\mathcal{G})\neq\gamma_{\mathcal{A}_2}(\mathcal{G})$,
  a contradiction.
\end{proof}

In the next proposition, we show that the automorphisms of
$\mathcal{G}$ preserve the structure of blocks in indiscernibility
partition.
\begin{prop}\label{8}
For a graph $\mathcal{G}$, let
$\mathcal{A}\subseteq\mathcal{V}(\mathcal{G})$ be any attribute
subset. If $\eta\in \Gamma(\mathcal{G})$ then
$\gamma_{\mathcal{A}}(\mathcal{G})=\gamma_{\eta(\mathcal{A})}(\mathcal{G})$.
\end{prop}
\begin{proof}
   Suppose $\gamma_{\mathcal{A}}(\mathcal{G})=\mathcal{B}_1|\mathcal{B}_2|\dots|\mathcal{B}_m$ and $\gamma_{\eta(\mathcal{A})}(\mathcal{G})=\mathcal{K}_1|\mathcal{K}_2|\dots|\mathcal{K}_n$.
Let $\mathcal{B}_i=\mathcal{O}_{\mathcal{A}}(x_i)=\{x_j:\mathcal{O}(x_j)\cap \mathcal{A}=\mathcal{O}(x_i)\cap \mathcal{A}\}$ and $\mathcal{K}_i =\mathcal{O}_{\eta({\mathcal{A}})}(x_i)=\{x_j:\mathcal{O}(x_j)\cap \eta(\mathcal{A})=\mathcal{O}(x_i)\cap \eta(\mathcal{A})\}$.
\par
It is easy to see that for $\eta \in \Gamma(\mathcal{G})$,  $\mathcal{O}(x_i)\cap \mathcal{A}\neq \emptyset$ if and only if $\mathcal{O}(x_i)\cap \eta(\mathcal{A}) \neq \emptyset$ for $1\leq i \leq k$ because $x_k \in \mathcal{O}(x_i)\cap \mathcal{A}$ if and only if $\eta(x_k) \in \mathcal{O}(x_i)\cap \eta(\mathcal{A})$. By \ref{22}, $\mathcal{B}_i=\mathcal{K}_i=\mathcal{O}(x_i)$ or $\mathcal{B}_i=\mathcal{K}_i=(\mathcal{O}(\mathcal{A}))^c$ which implies that $\gamma_{\mathcal{A}}(\mathcal{G})=\gamma_{\eta(\mathcal{A})}(\mathcal{G})$.
\end{proof}

For a set of attributes $\mathcal{A}\subseteq \mathcal{V}$,
$IND_{\mathcal{A}}(\mathcal{V}) =\{(x,y)\in
\mathcal{V}^2|\mathcal{O}(x)\cap \mathcal{A}=\mathcal{O}(y)\cap
\mathcal{A}\}.$ An attribute $x \in\mathcal{A}$ is called
dispensable in $\mathcal{A}$ if
$IND_{\mathcal{A}}(\mathcal{V})=IND_{\mathcal{A}\setminus\{x\}}(\mathcal{V})$,
otherwise $x$ is called indispensable with respect to $\mathcal{A}$.
A minimal subset $\mathcal{R}$ of an attribute set $\mathcal{A}$
that yields the same partition as provided by the set of all
attributes is called $reduct$.

\begin{prop}  For a graph $\mathcal{G}$ with a reduct $\mathcal{R}$ and orbits $\mathcal{O}_1,\mathcal{O}_2,
    \ldots,\mathcal{O}_k$, $|\mathcal{R} \cap \mathcal{O}_i|=1$ for all $i=1,2,\ldots,k$ except one $i$.
\end{prop}
%\textbf{proof missing}
\begin{proof}
Suppose there exists two orbits for which $\mathcal{R}\cap
\mathcal{O}_{i}$ and $\mathcal{R}\cap \mathcal{O}_{j}$ is empty then
$\mathcal{O}_{i}\cup \mathcal{O}_{j}$ forms one class. Hence,
$\mathcal{R}$ must have empty intersection with at most one orbit.
Now, suppose that there exists an orbit $\mathcal{O}_{k}$ such that
$|\mathcal{R}\cap \mathcal{O}_{k}| > 1 $. Let $x_{1},\, x_{2}\in
\mathcal{R}\cap \mathcal{O}_{k}$ then
$\gamma_{\mathcal{R}}(\mathcal{I})=\gamma_{\mathcal{R}\setminus
\{x_{2}\}}(\mathcal{I})$. This implies $\mathcal{R}$ is not minimal (a
contradiction). Hence, $|\mathcal{R} \cap \mathcal{O}_i|=1$ for all
$i=1,2,\ldots,k$ except one $i$.
\end{proof}

% Reduct $R$ of a graph $\mathcal{G}$ is a collection of one member
%from each $k-1$ orbits of $\mathcal{G}$. Let $\mathcal{C}\subseteq
%\mathcal{A}$, then by definition of reduct of $\mathcal{G}$,
%$\gamma_{\mathcal{A}}=\gamma_{\mathcal{C}}$ and
%$\gamma_{\mathcal{C}\setminus\{c\}}\neq \gamma_{\mathcal{A}}\,
%\forall \,  c \in \mathcal{C}$, then $\mathcal{C}$ is a reduct i.e.,
%$\mathcal{C}=\mathcal{R}$. Now, let us assume contrary that
%$\mathcal{R}$ is the collection of at least one member of each orbit
%except one. In particular, there does not exists any member from
%exactly one orbit in $\mathcal{R}$ i.e., $\mathcal{R}\cap
%\mathcal{O}_{j}=\emptyset \Rightarrow$ $|\mathcal{R}\cap
%\mathcal{O}_{j}|=0$ where $j$ is any one of the member of $i=1,
%2,...,k$. Therefore it yields that $|\mathcal{R}\cap
%\mathcal{O}_{i}| > 1 \, \forall \, i=1,2,...,k$ except the one
%orbit, (a contradiction). Hence, $|\mathcal{R} \cap
%\mathcal{O}_i|=1$ for all $i=1,2,\ldots,k$ except one.

It is obvious that any attribute subset $\mathcal{A}$ may have more
than one reducts, the set of all reducts of an attribute subset is
denoted by $\mathcal{RED}_{\mathcal{A}}(\mathcal{G})$. Let $\mathcal{K}\in\mathcal{RED}_{\mathcal{A}}(\mathcal{G})$ then
$\gamma_{\mathcal{A}}(\mathcal{G})=\gamma_{\mathcal{K}}(\mathcal{G})=\mathcal{B}_1|
\mathcal{B}_2|\dots|\mathcal{B}_n$. By proposition \ref{8},
$\gamma_{\eta(\mathcal{A})}(\mathcal{G})=\gamma_{(\mathcal{\eta(K)})}(\mathcal{G})
=\mathcal{B}_1|\mathcal{B}_2|\dots|\mathcal{B}_n$. We note that the automorphism of
$\mathcal{G}$ preserves the structure of reducts. Hence we have the following proposition:
\begin{prop}\label{9}
For a graph $\mathcal{G}$, let $\mathcal{A}\subseteq\mathcal{V}$. If
$\eta\in\Gamma(\mathcal{G})$, then
$\mathcal{K}\in\mathcal{RED}_{\mathcal{A}}(\mathcal{G})$ if and only
if
$\eta(\mathcal{K})\in\mathcal{RED}_{\eta(\mathcal{A})}(\mathcal{G})$.

\end{prop}

According to Proposition \ref{9}, for a graph $\mathcal{G}$,
attribute set $\mathcal{A}$ and $\eta\in\Gamma(\mathcal{G})$,
$\mathcal{K}\in\mathcal{RED}_{\mathcal{A}}(\mathcal{G})$ implies
that
$\eta(\mathcal{K})\in\mathcal{RED}_{\eta(\mathcal{A})}((\mathcal{G}))$.
If $\mathcal{K}$ is a reduct for attribute set $\mathcal{A}$ then
$\eta(\mathcal{K})$ does not necessarily belong to
$\mathcal{RED}_{\mathcal{A}}(\mathcal{G})$.
\begin{figure}[h]
  \centering
  \includegraphics[width=5cm]{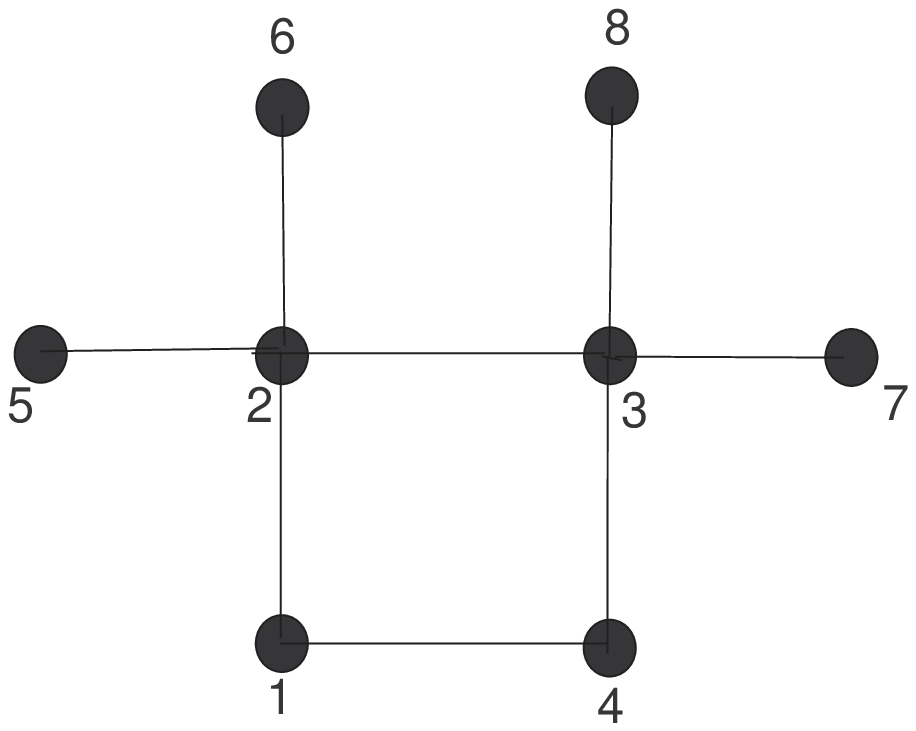}
  \caption{A Graph $\mathcal{G}$}\label{1}
\end{figure}

Consider a graph $\mathcal{G}$ in FIGURE 2. Note that
$\mathcal{O}(1)=\mathcal{O}(4)=\{1,4\}$,
$\mathcal{O}(2)=\mathcal{O}(3)=\{2,3\}$ and
$\mathcal{O}(5)=\mathcal{O}(6)=\mathcal{O}(7)=\mathcal{O}(8)=\{5,6,7,8\}$.
Let $\mathcal{A}=\{1,2,5,6\}$ and take $\eta=(14)(23)(57)(68)$, then
$\mathcal{K}=\{1,2\}$ is a reduct of $\mathcal{A}$ and
$\eta(\mathcal{K})=\{3,4\}$ is a  reduct of $\eta(\mathcal{A})$ but
not that of $\mathcal{A}$.

\par
Let $\texttt{P}(\mathcal{V})$ denotes the set of all partitions of
$\mathcal{V}$.  For two partitions $\gamma_1$ and $\gamma_2$ of
$\mathcal{V}$, $\gamma_1$ is called finer than $\gamma_2$ if all blocks of
$\gamma_1$ are subset of the blocks of $\gamma_2$, denoted by
$\gamma_1\preceq \gamma_2$. It is straightforward to see that for two attribute sets,
$\mathcal{A}_1$ and $\mathcal{A}_2$ of graph $\mathcal{G}$ with
$\mathcal{A}_1\subseteq\mathcal{A}_2$, then
$\gamma_{\mathcal{A}_2}(\mathcal{G}) \preceq \gamma_{\mathcal{A}_1}(\mathcal{G})$.
For a graph $\mathcal{G}$ with orbits $\mathcal{O}_1, \mathcal{O}_2,
\ldots, \mathcal{O}_k$ and attribute sets $\mathcal{A}_1\subseteq
\mathcal{A}_2$,
$\gamma_{\mathcal{A}_2}(\mathcal{G})=\gamma_{\mathcal{A}_1}(\mathcal{G})$,
if $\mathcal{A}_1$ and $\mathcal{A}_2$ have non-empty intersection
with same orbits $\mathcal{O}_i$ for $1\leq i\leq k$. If
$\mathcal{A}_1\subseteq \mathcal{A}_2$ and there exists some
$\mathcal{O}_{i}$ $(1\leq i\leq k)$ such that $\mathcal{A}_1\cap
\mathcal{O}_{i}=\emptyset$ and $\mathcal{A}_2\cap \mathcal{O}_{i}\neq
\emptyset$ then
$\gamma_{\mathcal{A}_2}(\mathcal{G})\prec \gamma_{\mathcal{A}_1}(\mathcal{G})$.
From this discussion, we conclude that a partition consisting of
orbits of the graphs is the finest partition of the vertex set in
$\texttt{P}(\mathcal{V})$.

%\begin{defn}\label{13}
%For an attribute subset $\mathcal{A}\subseteq\mathbb{A}$, let
%$x\in\mathcal{V}$. Then
%\begin{itemize}
%\item $x$ is called $\mathcal{A}$-interior, if $\mathcal{O}(x)\subseteq\mathcal{A}$.
%\item $x$ is called $\mathcal{A}$-exterior, if $\mathcal{O}(x)\subseteq\mathcal{V}\backslash\mathcal{A}$.
%\item $x$ is called $\mathcal{A}$-delimiting, if $\mathcal{O}_{\mathcal{A}}(x)\neq\emptyset$ and
%$\mathcal{O}_{\mathcal{V}\backslash\mathcal{A}}(x)\neq\emptyset$.
%\end{itemize}
%Represent by $Int(\mathcal{A})$, $Ext(\mathcal{A})$ and
%$Del(\mathcal{A})$ respectively, the subset of all
%$\mathcal{A}$-interior, $\mathcal{A}$-exterior and
%$\mathcal{A}$-delimiting vertices of graph $\mathcal{G}$.
%\end{defn}

%\begin{prop}\label{14}
%Let $\mathcal{A}_1$ and $\mathcal{A}_2$ are two attribute subsets of
%graph $\mathcal{G}$ such that $\mathcal{A}_1\subseteq\mathcal{A}_2$,
%then
%$\gamma_{\mathcal{A}_2}(\mathcal{G})\succeq \gamma_{\mathcal{A}_1}(\mathcal{G})$.
%\end{prop}

Now, we introduce three sets, namely $\mathcal{A}$-interior set,
$\mathcal{A}$-exterior set and $\mathcal{A}$-delimiting set as
follows:

\begin{defn}\label{31}
For an attribute subset $\mathcal{A}\subseteq\mathbb{A}$, let
$x\in\mathcal{V}$. Then
\begin{itemize}
\item $x$ is called $\mathcal{A}$-interior, if $\mathcal{O}(x)\subseteq\mathcal{A}$.
\item $x$ is called $\mathcal{A}$-exterior, if $\mathcal{O}(x)\subseteq\mathcal{V}\backslash\mathcal{A}$.
\item $x$ is called $\mathcal{A}$-delimiting, if $\mathcal{O}_{\mathcal{A}}(x)\neq\emptyset$ and
$\mathcal{O}_{\mathcal{V}\backslash\mathcal{A}}(x)\neq\emptyset$.
\end{itemize}
Represent by $Int(\mathcal{A})$, $Ext(\mathcal{A})$ and
$Del(\mathcal{A})$ respectively, the subset of all
$\mathcal{A}$-interior, $\mathcal{A}$-exterior and
$\mathcal{A}$-delimiting vertices of graph $\mathcal{G}$.
\end{defn}

Note that the family of all subsets of definition \ref{31} is a set
partition of $\mathcal{V}(\mathcal{G})$ and for an attribute subset
$\mathcal{A}\subseteq\mathcal{V}(\mathcal{G})$, and
$\mathcal{A}^c=\mathcal{V}\backslash\mathcal{A}$ then
$Int(\mathcal{A})=Ext(\mathcal{A}^c)$ and
$Del(\mathcal{A})=Del(\mathcal{A}^c)$. For an attribute subset
$\mathcal{A}\subseteq\mathcal{V}(K_n)$, following is
straightforward.
\begin{description}
  \item[$(i)$] If $|\mathcal{A}|<n$ then $Int(\mathcal{A})=Ext(\mathcal{A})=\emptyset$ and
  $Del(\mathcal{A})=\mathcal{V}(K_n)$.
  \item[$(ii)$] If $|\mathcal{A}|=n$ then $Del(\mathcal{A})=Ext(\mathcal{A})=\emptyset$ and
  $Int(\mathcal{A})=\mathcal{V}(K_n)$.
  \item[$(iii)$] If $\mathcal{A}=\emptyset$ then $Int(\mathcal{A})=Del(\mathcal{A})=\emptyset$
  and $Ext(\mathcal{A})=\mathcal{V}(K_n)$.
\end{description}

\begin{prop}
For an attribute subset $\mathcal{A}\subseteq\mathbb{A}$, a
non-empty $Ext(\mathcal{A})$ is a block of
$\mathcal{A}$-indiscernible partition
$\gamma_{\mathcal{A}}(\mathcal{G})$.
\end{prop}
\begin{proof}
Let $x,x\prime\in Ext(\mathcal{A})$ then for all $y\in\mathcal{A}$,
we have $\mathcal{F}(x,y)=\mathcal{F}(x\prime,y)=0$, so
$x\equiv_{\mathcal{A}}x\prime$. If $z\in \mathcal{V}(\mathcal{G})$
such that for some $z\prime \in Ext(\mathcal{A})$,
$z\equiv_{\mathcal{A}}z\prime$ then
$\mathcal{O}_{\mathcal{A}}(z)=\emptyset$ implying that $z \in
Ext(\mathcal{A})$. Because if there exists a vertex
$y\in\mathcal{O}_{\mathcal{A}}(z)$ then
$1=\mathcal{F}(y,z)\neq\mathcal{F}(y,z\prime)=0$, which is
contradiction to the assumption that $z\equiv_{\mathcal{A}}z\prime$.
Therefore, $\mathcal{O}_{\mathcal{A}}(z)=\emptyset$ and $z\in
Ext(\mathcal{A})$.
\end{proof}

\begin{rem}
It is straightforward to observe that the $Int(\mathcal{A})$ and
$Del(\mathcal{A})$ need not be a block of
$\gamma_{\mathcal{A}}(\mathcal{G})$. Further, $Int(\mathcal{A})$ and
$Del(\mathcal{A})$ are blocks of $\gamma_{\mathcal{A}}(\mathcal{G})$
if and only if it consists of only one orbit.
\end{rem}
\begin{prop}
For an attribute subset $\mathcal{A}$, if $x,y\in
Int(\mathcal{A})\cup Del(\mathcal{A})$, then
$x\equiv_{\mathcal{A}}y$ if and only if
$\mathcal{O}(x)=\mathcal{O}(y)$.
\end{prop}
\begin{proof}
  Suppose that $x,y\in Int(\mathcal{A})\cup Del(\mathcal{A})$ and $\mathcal{O}(x)=\mathcal{O}(y)$,
  then by definition of indiscernibility, $x\equiv_{\mathcal{A}}y$. Conversely, suppose that
  $x\equiv_{\mathcal{A}}y$ and $\mathcal{O}(x)\neq \mathcal{O}(y)$. Then there exists $z \in \mathcal{A}$
  which will belong to at most one of $\mathcal{O}(x)$ or $\mathcal{O}(y)$ because if no such
  $z \in \mathcal{A}$ belongs to $\mathcal{O}(x)$ or $\mathcal{O}(y)$ then $x,y\in Ext(\mathcal{A})$.
  Therefore, $\urcorner(x\equiv_{\mathcal{A}}y)$, contradiction.
  Hence $\mathcal{O}(x)=\mathcal{O}(y)$.
\end{proof}

\section{Upper and Lower Approximations in Graphs}

In this section, lower and upper approximations of various subsets of vertices are studied using similarity relations.
\par

Theorem \ref{4}, provides information about the behavior of objects
in the indiscernibility relation $\equiv_{\mathcal{A}}$.
Indiscernibility relations can be considered as a type of symmetry
relations as for the attribute subset $\mathcal{A}$.
Using the fact that for any two
distinct vertices $x, y\in \mathcal V (\mathcal{G}) $ either $\mathcal{O}(x)=
\mathcal{O}(y)$ or $\mathcal{O}(x) \cap \mathcal{O}(y)=\emptyset$, it is straightforward to
note that every proper subset $X$ of
$\mathcal{V(\mathcal{G})}$ is  rough if $\mathcal{O}(x) =\mathcal{V(\mathcal{G})}$. Also
note that for $x, y \in \mathcal{V(\mathcal{G})}$ with $O(x)\neq O(y)$ and $O(x)\cup O(y) = \mathcal{V(\mathcal{G})}$, a
subset X of $\mathcal{V(\mathcal{G})}$ is rough if and only if $X
\neq \mathcal{O}(x), \, X \neq \mathcal{O}(y)$ and $X \neq
\mathcal{V(\mathcal{G})}$.

 Note that the indiscernibility partition of complete
graph on $n$ vertices consists of only one orbit yielding one block.
If $\mathcal{V}(K_n)=\{x_1,\dots,x_n\},$ then
$\gamma_{\mathcal{A}}(K_n)=x_1,\dots,x_n$ for any attribute subset
$\mathcal{A}\subseteq K_n$ and every proper subset $\mathcal{Q}$ of
$\mathcal{V}(K_n)$ is $\mathcal{A}$-rough because
$\mathbb{L}_\mathcal{A}(\mathcal{Q})=\emptyset$ and
$\mathbb{U}_\mathcal{A}(\mathcal{Q})=\mathcal{V}(K_n)$. A subset
$\mathcal{Q}$ of $\mathcal{V}(K_n)$ is $\mathcal{A}$-exact if and
only if $\mathcal{Q}=\mathcal{V}(K_n)$.

A given graph $\mathcal{G}$ in which $\mathcal{V}=\mathcal{V}_1\cup
\mathcal{V}_2$, $\mathcal{V}_1\cap \mathcal{V}_2=\emptyset$, $x\sim
y$ for each $x\in \mathcal{V}_1$, $y\in \mathcal{V}_2$ and $x\nsim
y$ if and only if $x,y\in \mathcal{V}_1$ or $x,y\in \mathcal{V}_2$,
is called a \emph{complete bipartite graph}. For a complete
bipartite graph $\mathcal{K}_{m,n}=(\mathcal{V}_1|\mathcal{V}_2)$
where $\mathcal{V}_1=\{x_1,x_2,\dots,x_m\}$ and
$\mathcal{V}_2=\{y_1,y_2,\dots,y_n\}$,
$\gamma_{\mathcal{A}}(\mathcal{K}_{m,n})=x_1,x_2,\dots,x_m|y_1,y_2,\dots,y_n$,
for any non-empty attribute subset $\mathcal{A}$. \\
In the following proposition, we study rough sets in complete bipartite graphs.
\begin{prop}\label{17}
    If $\mathcal{G}$ is a complete bipartite graph then a subset
    $\mathcal{Q}$ of $\mathcal{V}$ is rough with respect to attributes
    $\mathcal{A}$ if and only if $\mathcal{Q}\neq \mathcal{V}_1$, $\mathcal{Q}\neq \mathcal{V}_2$
    and $\mathcal{Q}\neq \mathcal{V}$, where $\mathcal{V}_1,\mathcal{V}_2$ are partites of $\mathcal{G}$.
\end{prop}

\begin{proof}
Let $\mathcal{Q}$ be a rough set in $\mathcal{G}$. If
$\mathcal{Q}=\mathcal{V}$ then clearly $\mathcal{Q}$ is exact, hence
$\mathcal{Q}\neq \mathcal{V}$. Now if $\mathcal{Q}=\mathcal{V}_1$
then for each $x\in \mathcal{V}_2$, $\mathcal{O}(x)\cap
\mathcal{Q}=\emptyset$ which gives that $\mathcal{Q}$ is exact.
Similar arguments holds for $\mathcal{Q}=\mathcal{V}_2$. Hence,
$\mathcal{Q}\neq \mathcal{V}_1$, $\mathcal{Q}\neq \mathcal{V}_2$ and
$\mathcal{Q}\neq \mathcal{V}$.

Conversely, suppose $\mathcal{Q}\neq \mathcal{V}_1$,
$\mathcal{Q}\neq \mathcal{V}_2$ and $\mathcal{Q}\neq \mathcal{V}$,
we need to prove that $\mathbb{L}_\mathcal{A}(\mathcal{Q})\neq \mathbb{U}_\mathcal{A}(\mathcal{Q})$.
Since $\mathcal{Q}\neq \mathcal{V}_1$ so we have three cases:\\
(1) Suppose $\mathcal{V}_1\cap \mathcal{Q}=\emptyset$ and
$\mathcal{V}_2\cap \mathcal{Q}\neq\emptyset$ then $\mathcal{Q}\neq
\mathcal{V}_2$ which implies that $\mathcal{Q}\subset
\mathcal{V}_2$. Now for any $x\in \mathcal{V}$, if $x\in
\mathcal{V}_1$ then $\mathcal{O}(x)\cap \mathcal{Q}=\emptyset$
because $\mathcal{O}(x)=\mathcal{V}_1$. Moreover, for $x\in
\mathcal{V}_2$, $\mathcal{O}(x)=\mathcal{V}_2$ which implies that
$\mathbb{L}_\mathcal{A} (\mathcal{Q})=\emptyset$ and
$\mathbb{U}_\mathcal{A} (\mathcal{Q})
     =\mathcal{V}_2$. Hence, $\mathcal{Q}$ is a rough set.\\
(2) For the case $\mathcal{V}_1\cap \mathcal{Q}\neq\emptyset$ and
$\mathcal{V}_2\cap \mathcal{Q}=\emptyset$ then $\mathcal{Q}\neq
\mathcal{V}_1$ which implies that $\mathcal{Q}\subset
\mathcal{V}_1$. Now for any $x\in \mathcal{V}$, if $x\in
\mathcal{V}_1$ then $\mathcal{O}(x)\cap \mathcal{Q}\neq\emptyset$
because $\mathcal{O}(x)=\mathcal{V}_1$. Moreover, for $x\in
\mathcal{V}_2$, $\mathcal{O}(x)=\mathcal{V}_2$ which implies that
$\mathbb{L}_\mathcal{A} (\mathcal{Q})=\emptyset$ and
$\mathbb{U}_\mathcal{A} (\mathcal{Q})
    =\mathcal{V}_1$. Hence, $\mathcal{Q}$ is a rough set.\\
(3) Now suppose $\mathcal{V}_1\cap \mathcal{Q}\neq\emptyset$ and
$\mathcal{V}_2\cap \mathcal{Q}\neq\emptyset$
then we have three sub-cases:\\
i) Suppose $\mathcal{V}_1\cap \mathcal{Q}=\mathcal{V}_1$ then
$\mathbb{L}_\mathcal{A} (\mathcal{Q})= \mathcal{V}_1$
and $\mathbb{U}_\mathcal{A} (\mathcal{Q})=\mathcal{V}$ which gives that $\mathcal{Q}$ is rough.\\
ii) Suppose $\mathcal{V}_2\cap \mathcal{Q}=\mathcal{V}_2$ then for
every $x\in \mathcal{V}_2$ we have $x\in \mathbb{L}_\mathcal{A}
(\mathcal{Q})= \mathcal{V}_2$ and $\mathbb{U}_\mathcal{A}
(\mathcal{Q})=\mathcal{V}$
which gives that $\mathcal{Q}$ is rough.\\
iii) Suppose $\mathcal{V}_1\cap \mathcal{Q}\neq\mathcal{V}_1$ and
$\mathcal{V}_2\cap \mathcal{Q}\neq\mathcal{V}_2$. Let $x\in
\mathcal{V}$ such that $x\in \mathcal{V}_1$ then $x\in
\mathbb{U}_\mathcal{A}(\mathcal{Q})$ because
$\mathcal{O}(x)=\mathcal{V}_1$ and $\mathcal{O}(x)\cap
\mathcal{Q}\neq \emptyset$ but $\mathcal{O}(x)\nsubseteq
\mathcal{Q}$. Similarly, if $x\in \mathcal{V}_2$ then $x\in
\mathbb{U}_\mathcal{A}(\mathcal{Q})$. Also, for every $x\in
\mathcal{V}$ we have $\mathcal{O}(x)\nsubseteq \mathcal{Q}$ which
gives that $\mathbb{L}_\mathcal{A}(\mathcal{Q})=\emptyset$ so
$\mathcal{Q}$ is a rough set in $\mathcal{G}$.
\end{proof}

In the following proposition, we discuss lower and upper
approximation of a subset $\mathcal{Q}$ of complete bipartite graph
$\mathcal{K}_{m,n}$.

\begin{prop}\label{5}
Let $\mathcal{K}_{m,n}=(\mathcal{V}_1|\mathcal{V}_2)$ is a complete
bipartite graph where $\mathcal{V}_1=\{x_1,x_2,\dots,x_m\}$ and
$\mathcal{V}_2=\{y_1,y_2,\dots,y_n\}$. Let $\mathcal{Q}$ and
$\mathcal{A}$ are two subsets of $\mathcal{V}(\mathcal{K}_{m,n})$
such that $\mathcal{Q}\neq\mathcal{V}$. Then
\begin{enumerate}
  \item  \begin{center}
  $\mathbb{L}_{\mathcal{A}}(\mathcal{Q})=\left\{\begin{array}{lll}
  \mathcal{V}_1 & \mbox{if} & \mathcal{V}_1\subseteq\mathcal{Q}\\
  \mathcal{V}_2 & \mbox{if} & \mathcal{V}_2\subseteq\mathcal{Q}\\
  \emptyset & \mbox{ } & otherwise.
  \end{array}\right.$
  \end{center}
  \item \begin{center}
  $\mathbb{U}_{\mathcal{A}}(\mathcal{Q})=\left\{\begin{array}{lll}
  \mathcal{V}_1 & \mbox{if} & \mathcal{Q}\subseteq \mathcal{V}_1\\
  \mathcal{V}_2 & \mbox{if} & \mathcal{Q}\subseteq \mathcal{V}_2\\
  \mathcal{V} & \mbox{ } & otherwise.
  \end{array}\right.$
  \end{center}
  \item \begin{center} $\mathcal{Q}$ is $\mathcal{A}$-exact if and only if
  $\mathcal{Q}=\mathcal{V}_1$ or $\mathcal{V}_2$.
  \end{center}
\end{enumerate}
\end{prop}
\begin{proof}\noindent
\begin{enumerate}
 \item Let $\mathcal{V}_1\subseteq\mathcal{Q}$. If $x\in \mathcal{V}_1$, then it follow by Theorem
 1.0.1
  that $\mathcal{O}_{\mathcal{A}}(x)\subseteq \mathcal{Q}$,
thus by definition of the lower approximation, we get
$\mathcal{V}_1\subseteq\mathbb{L}_{\mathcal{A}}(\mathcal{Q})$.
Moreover, if $x\in\
\mathcal{V}_2\cap\mathbb{L}_{\mathcal{A}}(\mathcal{Q})$, for some
$x\in\mathcal{V}$, then, again by Theorem 1.0.1 and definition of
the lower approximation, we get
$\mathcal{V}_2=\mathcal{O}_{\mathcal{A}}(x)\subseteq \mathcal{Q}$.
Because $\mathcal{V}_1|\mathcal{V}_2$ are set partition for
$\mathcal{V}$, so last inclusion implies $\mathcal{Q}=\mathcal{V}$,
a contradiction to our assumption. Thus,
$\mathcal{V}_1\subseteq\mathbb{L}_{\mathcal{A}}(\mathcal{Q})$ and
$\mathcal{V}_2\cap\mathbb{L}_{\mathcal{A}}(\mathcal{Q})=\emptyset$.
Similarly, if $\mathcal{V}_2\subseteq\mathcal{Q}$ then
$\mathbb{L}_{\mathcal{A}}(\mathcal{Q})=\mathcal{V}_2$. Hence let
$\mathcal{V}_1\nsubseteq\mathcal{Q}$ and
$\mathcal{V}_2\nsubseteq\mathcal{Q}$. but each vertex
$x\in\mathcal{V}$ is either in $\mathcal{V}_1$ or $\mathcal{V}_2$.
Therefore, by Theorem 1.0.1, we have
$\mathcal{O}_{\mathcal{A}}(x)=\mathcal{V}_1\nsubseteq\mathcal{Q}$
and
$\mathcal{O}_{\mathcal{A}}(x)=\mathcal{V}_2\nsubseteq\mathcal{Q}$,
$i.e.,$ $x\notin\mathbb{L}_{\mathcal{A}}(\mathcal{Q})$. Hence
$\mathbb{L}_{\mathcal{A}}(\mathcal{Q})=\emptyset$.
 \item Let $\mathcal{Q}\subseteq \mathcal{V}_1$. If $x\in \mathcal{V}_1$, then it follow by
 Theorem 1.0.1 that
 $\mathcal{O}_{\mathcal{A}}(x)=\mathcal{V}_1\cap\mathcal{Q}\neq\emptyset$,
 because $\mathcal{Q}$ is non-empty subset of $Q_1$. Hence
 $x\in\mathbb{U}_{\mathcal{A}}(\mathcal{Q})$. Moreover, if
 $x\in\mathbb{U}_{\mathcal{A}}(\mathcal{Q})$ then by definition of upper
 approximation we get
 $\mathcal{O}_{\mathcal{A}}(x)\cap\mathcal{Q}\neq\emptyset$. Let
 $y\in\mathcal{O}_{\mathcal{A}}(x)\cap\mathcal{Q}$. Since $y\in\mathcal{Q}\subseteq
 \mathcal{V}_1$, thus by Theorem 1.0.1 we have
 $\mathcal{V}_1=\mathcal{O}_{\mathcal{A}}(x)=\mathcal{O}_{\mathcal{A}}(y)$,
 therefore again by Theorem 1.0.1 we deduce that $x\in \mathcal{V}_1$. Hence
 $\mathbb{U}_{\mathcal{A}}(\mathcal{Q})=\mathcal{V}_1$. The case of $\mathcal{Q}\subseteq
 \mathcal{V}_2$ is similar. Finally, $\mathcal{Q}\nsubseteq \mathcal{V}_1$ and $\mathcal{Q}\nsubseteq
 \mathcal{V}_2$. Since $\mathcal{V}_1|\mathcal{V}_2$ is a partition of $\mathcal{V}$, which
 implies that $\mathcal{V}_1\cap\mathcal{Q}\neq\emptyset$ and
 $\mathcal{V}_2\cap\mathcal{Q}\neq\emptyset$. Now select arbitrary a vertex
 $x\in\mathcal{V}$, then either $x\in \mathcal{V}_1$ or $\mathcal{V}_2$. If $x\in
 \mathcal{V}_1$, then by Theorem 1.0.1 it follows that
 $\mathcal{O}_{\mathcal{A}}(x)\cap\mathcal{Q}=\mathcal{V}_1\cap\mathcal{Q}\neq\emptyset$.
 Thus, $x\in\mathbb{U}_{\mathcal{A}}(\mathcal{Q})$. Analogously, if
 $x\in \mathcal{V}_2$. Then it shows that
 $\mathcal{V}\subseteq\mathbb{U}_{\mathcal{A}}(\mathcal{Q})$.
 $i.e.,$ $\mathcal{V}=\mathbb{U}_{\mathcal{A}}(\mathcal{Q})$.
 \item This follow from the definition of exactness and Theorem
 1.0.1.
 \end{enumerate}
 \end{proof}

 In proposition \ref{10}, we provide complete description for the lower
approximations of graph $\mathcal{G}$.
\begin{prop}\label{10}
For a graph $\mathcal{G}$, let
$\mathcal{A},\mathcal{Q}\subseteq\mathcal{V}$. Then
\begin{center}
$\mathbb{L}_{\mathcal{A}}(\mathcal{Q})= \left\{\begin{array}{lll}
\mathcal{O}(\mathcal{A}) & \mbox{where} &
x\in\mathcal{A}\wedge(\mathcal{O}(x)\setminus
\mathcal{Q})=\emptyset\\
\emptyset & \mbox{ } & otherwise
\end{array}\right.$
\end{center}
\end{prop}
\begin{proof}
Let a vertex $x\in\mathcal{A}$ such that
$\mathcal{O}_{\mathcal{A}}(x)\setminus \mathcal{Q}=\emptyset$, which
give $\mathcal{O}_{\mathcal{A}}(x)\subseteq \mathcal{Q}$, implies
that
$\mathcal{O}_{\mathcal{A}}(x)\in\mathbb{L}_{\mathcal{A}}(\mathcal{Q})$.
Since $x$ is chosen arbitrary, hence $\mathbb{L}_{\mathcal{A}}(\mathcal{Q})=\mathcal{O}(\mathcal{A})$.\\
For $\mathcal{O}_{\mathcal{A}}(x)\setminus
\mathcal{Q}\neq\emptyset$,
 $\mathcal{O}_{\mathcal{A}}(x)\nsubseteq \mathcal{Q}$ giving that
$\mathcal{O}_{\mathcal{A}}(x)\notin\mathbb{L}_{\mathcal{A}}(\mathcal{Q})$
implies that $\mathbb{L}_{\mathcal{A}}(\mathcal{Q})=\emptyset$.
\end{proof}

In proposition \ref{11}, we provide a necessary and sufficient
condition for a subset $\mathcal{Q}\subseteq\mathcal{V}$ to be
exact.
 \begin{prop}\label{11}
For a graph $\mathcal{G}$, let
$\mathcal{A},\mathcal{Q}\subseteq\mathcal{V}$. Then $\mathcal{Q}$ is
exact with respect to $\mathcal{A}$, if and only if
$\mathcal{Q}=\cup_{x\in \mathcal{Q}}\mathcal{O}_{\mathcal{A}}(x)$.
\end{prop}
\begin{proof} Suppose $\mathcal{Q}$ is exact then for $x \in \mathcal{Q}$,
$\mathcal{O}_{\mathcal{A}}(x)\subseteq \mathcal{Q}$. As $x$ is arbitrary,
$\cup_{x\in \mathcal{Q}}\mathcal{O}_{\mathcal{A}}(x)\subseteq\mathcal{Q}$.
For each $x\in \mathcal{Q}$, $x \in \mathcal{O}_{\mathcal{A}}(x)$ which yields that
$\mathcal{Q} \subseteq \cup_{x\in \mathcal{Q}}\mathcal{O}_{\mathcal{A}}(x)$.
Hence $\mathcal{Q}=\cup_{x\in \mathcal{Q}}\mathcal{O}_{\mathcal{A}}(x)$. Conversely,
suppose $\mathcal{Q}=\cup_{x\in \mathcal{Q}}\mathcal{O}_{\mathcal{A}}(x)$ which
clearly implies $\mathbb{L}_{\mathcal{A}}(\mathcal{Q})= \mathbb{U}_{\mathcal{A}}(\mathcal{Q})$
giving that $\mathcal{Q}$ is exact.
\end{proof}

A graph $\mathcal{G}$ in which no two vertices are similar is called
$\textit{rigid}$ graph. In such graphs,
$\gamma_{\mathcal{A}}(\mathcal{G})$ consists of blocks having
singleton elements only, which implies
$\mathbb{L}_{\mathcal{A}}(\mathcal{Q})
=\mathbb{U}_{\mathcal{A}}(\mathcal{Q})$. Hence every subset of the
vertex set of a rigid graph is exact.

\subsection{Rough Membership Function and Dependency}
Now, we introduce and present results on rough membership function, positive region and degree of dependency of
graphs using orbits of the graphs.

Let $\mathcal{G}=(\mathcal{V},\mathcal{E})$ and $\mathcal{A},\mathcal{Q}\subseteq\mathcal{V}$. The rough membership
function on the vertex set $\mathcal{V}$ is defined by:
\begin{center}
$\mu_{\mathcal{Q}}^{\mathcal{A}}(x)=\frac{|\{y\in \mathcal{Q}:
\mathcal{O}(x)\cap \mathcal{A}=\mathcal{O}(y)\cap
\mathcal{A}\}|}{|\{y\in \mathcal{V}: \mathcal{O}(x)\cap
\mathcal{A}=\mathcal{O}(y)\cap \mathcal{A}\}|}.$
\end{center}
For $\mathcal{A},\mathcal{Q}\subseteq \mathcal{V}$,
$POS_{\mathcal{A}}(\mathcal{Q})=\{x\in \mathcal{V}: (y\in
\mathcal{V} \wedge \mathcal{O}(x)\cap \mathcal{A}=
\mathcal{O}(y)\cap \mathcal{A})\Rightarrow (\mathcal{O}(x)\cap
\mathcal{Q}=\mathcal{O}(y)\cap \mathcal{Q})\}$. The number
$deg_{\mathcal{A}}({\mathcal{Q}})=\frac{|POS_{\mathcal{A}}({\mathcal{Q}})|}{|\mathcal{V}|}$
is referred to as $\mathcal{A}$-degree dependency of $\mathcal{Q}$. It is easy to note that
for two different attribute sets $\mathcal{A}$ and $\mathcal{Q}$ which yield the same partition,
$deg_{\mathcal{A}}({\mathcal{Q}})=1$. Further, if $\mathcal{G}$ has $k$ orbits and $\mathcal{A}$
has non-empty intersection with $k$ or $k-1$ orbits of $\mathcal{G}$, $deg_{\mathcal{A}}({\mathcal{V}})=1$.
\par
In proposition \ref{12}, we examine the rough membership function of
complete bipartite graph $\mathcal{K}_{m,n}$ and compute
$\mathcal{A}$-positive region of $\mathcal{Q}$ and the degree of
dependency.

\begin{prop}\label{12}
For a complete bipartite graph $\mathcal{K}_{m,n}$ with bipartition
$(\mathcal{V}_1|\mathcal{V}_2)$, let $\mathcal{A}$ and $\mathcal{Q}$
are two subset of $\mathcal{V}(\mathcal{K}_{m,n})$. Then
\begin{enumerate}
\item [$(i)$.]
\begin{center}
$\mu_{\mathcal{Q}}^{\mathcal{A}}(x)=\left\{\begin{array}{lll}
|\mathcal{V}_1\cap\mathcal{Q}|/|\mathcal{V}_1| & \mbox{if} & x\in \mathcal{V}_1\\
|\mathcal{V}_2\cap\mathcal{Q}|/|\mathcal{V}_2| & \mbox{if} & x\in
\mathcal{V}_2
\end{array}\right.$
\end{center}
\end{enumerate}
\begin{enumerate}
\item [$(ii)$.]
\begin{center}
$POS_{\mathcal{A}}({\mathcal{Q}})=\left\{\begin{array}{lll}
 \emptyset & \mbox{if} &
 \mathcal{A}=\emptyset\wedge\mathcal{Q}\neq\emptyset\\
 \mathcal{V} & \mbox{ } & otherwise
 \end{array}\right.$

\end{center}
\end{enumerate}
\begin{enumerate}
\item [$(iii)$.]
\begin{center}
$deg_{\mathcal{A}}(\mathcal{Q})=\left\{\begin{array}{lll}
 0 & \mbox{if} &
 \mathcal{A}=\emptyset\wedge\mathcal{Q}\neq\emptyset\\
 1 & \mbox{ } & otherwise
 \end{array}\right.
$
\end{center}
\end{enumerate}
\end{prop}

\begin{proof}
$(i)$. According to proposition \ref{5}, we know that
$\mathcal{O}_{\mathcal{A}}(x)=\mathcal{V}_i$, if and only if $x\in
\mathcal{V}_i$, for $i$=1,2, therefore,
\begin{center}
$\mathcal{O}_{\mathcal{A}}(x)\cap\mathcal{Q}=\left\{\begin{array}{lll}
|\mathcal{V}_1\cap\mathcal{Q}| & \mbox{if} & x\in \mathcal{V}_1\\
|\mathcal{V}_2\cap\mathcal{Q}| & \mbox{if} & x\in \mathcal{V}_2
\end{array}\right.$
\end{center}
Hence the proof follow directly by definition of rough membership.\\
$(ii)$. If $\mathcal{A}=\emptyset$ and $\mathcal{Q}\neq\emptyset$
then $\gamma_{\mathcal{A}}(\mathcal{K}_{m,n})=\mathcal{V}$. As
$\gamma_{\mathcal{Q}}(\mathcal{K}_{m,n})=\mathcal{V}_1|\mathcal{V}_2$.
Therefore, $\mathcal{O}_{\mathcal{A}}(x)=\mathcal{V}$ and
$\mathcal{O}_{\mathcal{Q}}(x)=\mathcal{V}_i$ for some $i=1,2$ which
implies that
$\mathcal{O}_{\mathcal{A}}(x)\nsubseteq\mathcal{O}_{\mathcal{Q}}(x)$
$\forall x\in\mathcal{V}$, hence
$POS_{\mathcal{A}}(\mathcal{Q})=\emptyset$ if
$\mathcal{A}=\mathcal{Q}=\emptyset$ then
$\gamma_{\mathcal{A}}(\mathcal{K}_{m,n})=\gamma_{\mathcal{Q}}(\mathcal{K}_{m,n})=\mathcal{V}$,
therefore
$\mathcal{O}_{\mathcal{A}}(x)=\mathcal{V}\subseteq\mathcal{O}_{\mathcal{Q}}(x)=\mathcal{V}$
$\forall$ $x\in\mathcal{V}$. Hence
$POS_{\mathcal{A}}(\mathcal{Q})=\mathcal{V}$. If
$\mathcal{A}\neq\emptyset$ and $\mathcal{Q}=\emptyset$ then
$\gamma_{\mathcal{A}}(\mathcal{K}_{m,n})=\mathcal{V}_1|\mathcal{V}_2$
by Proposition \ref{5}, and
$\gamma_{\mathcal{Q}}(\mathcal{K}_{m,n})=\mathcal{V}$. Which implies
that $\mathcal{O}_{\mathcal{A}}(x)=\mathcal{V}_i$ for some $i=1,2$
and $\mathcal{O}_{\mathcal{Q}}(x)=\mathcal{V}$ for all
$x\in\mathcal{V}$. Hence
$POS_{\mathcal{A}}(\mathcal{Q})=\mathcal{V}$. Finally, if
$\mathcal{A}\neq\emptyset$ and also $\mathcal{Q}\neq\emptyset$ then
by Proposition \ref{5}, we have
$\gamma_{\mathcal{A}}(\mathcal{K}_{m,n})=\gamma_{\mathcal{Q}}(\mathcal{K}_{m,n})=\mathcal{V}_1|\mathcal{V}_2$.
This implies that
$\mathcal{O}_{\mathcal{A}}(x)=\mathcal{O}_{\mathcal{Q}}(x)=\mathcal{V}_i$
for all $x\in\mathcal{V}$ and for some $i$=1,2. Thus in this case we
have $POS_{\mathcal{A}}(\mathcal{Q})=\mathcal{V}$.\\
$(iii)$. It is following directly from definition and $(ii)$.
\end{proof}

%Next, we discuss the $\mathcal{A}$-degree dependency for path graph,
%complete graph and complete bipartite graphs and also state the
%general interpretation for an arbitrary graph.

Note that, for complete graphs $K_{n}$ and cycle graphs $C_{n}$,
$|POS_{\mathcal{A}}({\mathcal{Q}})|=|\mathcal{V}|$ yields that
$deg_{\mathcal{A}}({\mathcal{Q}})=1$. It is observed that for
$K_{n}$, $C_{n}$, we have $\gamma_{\mathcal{V}}=\mathcal{O}_{1}$ and
for complete bipartite graphs $\mathcal{K}_{m,n}$, we have
$\gamma_{\mathcal{V}}=\mathcal{O}_{m}|\mathcal{O}_{n}$. Therefore,
$deg_{\mathcal{A}}({\mathcal{Q}})=1$ when $m=n$ and
$deg_{\mathcal{A}}({\mathcal{Q}})< 1$ when $m\neq n$. For a path
graph $P_{n}$, with $k$ orbits and  $\mathcal{A},
\mathcal{Q}\subseteq \mathcal{V}$ such that $ \mathcal{A}\subseteq
\mathcal{Q}$ and $\mathcal{A}\cap \mathcal{O}_{k-i}$ and
$\mathcal{Q}\cap \mathcal{O}_{k-1}$ is non-empty $\forall \, i > 1$
then $deg_{\mathcal{A}}(\mathcal{Q})< 1$ and
$deg_{\mathcal{Q}}(\mathcal{A})= 1$.

\par It is easy to see that for a graph $\mathcal{G}$ with $\mathcal{A},\mathcal{Q}\subseteq\mathcal{V}$,
$\mu_{\mathcal{Q}}^{\mathcal{A}}(x)=0$ if $\gamma(x)\cap
\mathcal{Q}= \emptyset$, $\mu_{\mathcal{Q}}^{\mathcal{A}}(x)< 1$ if
$\gamma(x)\cap \mathcal{Q}\neq \emptyset$ and
$\mu_{\mathcal{Q}}^{\mathcal{A}}(x)=1$ if $\gamma(x)\subset
\mathcal{Q}$ where $\gamma(x)$ is the block of partition $\gamma$
containing $x$.

\section{Essential Sets and Discernibility Matrix of Graphs}

In this section, we will introduce essential sets as well as discernibility matrices of graphs. We will also study the relationship between these two important concepts.
\subsection{Essential Sets of Graphs} Chiaselotti et al. \cite{7} presented classical model of Pawlak's core in
a more general way. The core of an information system $\mathcal{I}$
associated to $\mathcal{G}$ is the intersection of all reducts,
represented by $\mathcal{CORE}(\mathcal{G})$.
Hence removal of any attribute belonging to core of a graph leads to change in the indiscernibility partition.
In case, core is empty, this
extension becomes important because it provides with a set on minimum number of vertices
whose removal yields a partition different from the one yielded by all attributes. Concepts of definition \ref{33} were introduced by Chiaselotti et al. in \cite{7}.
\begin{defn}\label{33}

A subset $\mathcal{S}\subseteq \mathcal{A}$, is called
$\mathcal{I}$-essential of $\mathcal{G}$, if
$\gamma_{\mathcal{A}\setminus\mathcal{S}}(\mathcal{G})\neq\gamma_{\mathcal{A}}(\mathcal{G})$
and $\forall \, \mathcal{Q}\subsetneq\mathcal{S}$ we have
$\gamma_{\mathcal{A}\setminus\mathcal{Q}}(\mathcal{G})=\gamma_{\mathcal{A}}(\mathcal{G}).$
\end{defn}

$\mathsf{ESS}(\mathcal{I})$ represents the collection of all
$\mathcal{I}$-essential subsets of $\mathcal{G}$. For
$l\in\{1,2,\dots,n\}$, set
\begin{center}
$\mathsf{ESS}_l(\mathcal{I})=\{\mathcal{S}\in\mathsf{ESS}(\mathcal{I}):
|\mathcal{S}|=l\}$
\end{center}
and essential numerical sequence of $\mathcal{I}$ is defined by
\begin{center}
$ens(\mathcal{I})=(|\mathsf{ESS}_1(\mathcal{I})|,|\mathsf{ESS}_2(\mathcal{I})|,
\dots,|\mathsf{ESS}_n(\mathcal{I})|)$.
\end{center}
Finally, essential dimension of $\mathcal{I}$ is defined as the
positive integer
\begin{center}
$\mathsf{E}dim(\mathcal{I})=min\{l:|\mathsf{ESS}_l(\mathcal{I})|\neq
0\}$. %\cite{7}
\end{center}
Note that, for a graph with at most two orbits,
$ESS(\mathcal{I})=\emptyset$. For complete graphs $K_{n}$, cycles
$C_{n}$ and complete bipartite graphs $\mathcal{K}_{m,n}$,
$ESS(\mathcal{I})=\emptyset$. As there exists only one orbit for
$K_{n}$ and $C_{n}$, but for  $\mathcal{K}_{m,n}$, there exists one
orbit when $m=n$ and two orbits when  $m\neq n$. Also, it is
observed that if $\mathcal{A}\subseteq \mathcal{V}$, then there does
not exist any set $\mathcal{S}$ for complete graphs, cycles and complete
bipartite graphs such that $\gamma_{\mathcal{A}}(\mathcal{I})\neq
\gamma_{\mathcal{A}\backslash\mathcal{S}}(\mathcal{I})$. Therefore,
$ens(K_{n})=(0, 0, 0,...,0)$ $\forall \, n$, $ens(C_{n})=(0, 0,
0,...,0)$  $\forall \, n$ and $ens(\mathcal{K}_{m,n})=(0, 0,
0,...,0)$ $\forall \, m, \, n$.
%\par
%

\par For a path $P_n$ on $n\geq 5$ vertices with vertex set
$\mathcal{V}=\{x_1,x_2, \ldots,x_n\}$ and edge set
$\mathcal{E}=\{{x_ix_{i+1}}|1\leq i \leq n-1\}$. It is observed that
$\{\{x_i,x_{n-i+1}\}|1\leq i \leq \lfloor\frac{n}{2}\rfloor\}$ is the set of orbits of
$P_n$. Note that each orbit has at most two elements. It can be
easily seen that $ESS_l(\mathcal{I})=\emptyset$ for $l<3$.
For odd $n$, $\{x_{\frac{n+1}{2}}\}$ forms an orbit on one vertex and $ESS_{3}(\mathcal{I})$
is the union of orbit $\{x_{\frac{n+1}{2}}\}$ with any other
orbit of $P_{n}$. Also, $ESS_{4}(\mathcal{I})$ is the union of any
two orbits of $P_{n}$.
Further, observe that if $n=2k$ for $k>2$, then each orbit of
path graph $P_n$ will consist of exactly two vertices. Therefore,
${E}dim(P_n)=4$. Similarly, if $n=2k+1$ for $k\geq2$, then
$P_n$ will contain exactly one orbit of order one and all other
orbits of order two. Then ${E}dim(P_n)=3$.  Also
note that, for a rigid graph, ${E}dim(\mathcal{G})=2$ as
the cardinality of each orbit of a rigid graph is one.
\\
%It is also worth noting
%that for $k\geq 3$,$|ESS_{2k-1}(\mathcal{I})|=|ESS_{2k}(\mathcal{I})|$, $ens(P_n)=(0,0,0,...)$.
%\\ In the next two propositions, we state
%essential dimension of path graphs and rigid graphs, respectively.

%\begin{rem}
%For a path graph of order $n \geq 5$ with vertex set
%$\mathcal{V}=\{x_1,x_2, \ldots,x_n\}$ and edge set
%$\mathcal{E}=\{{x_ix_{i+1}}|1\leq i \leq n-1\}$. It is observed that
%$\{x_i,x_n-i+1\}|1\leq i \leq \frac{n}{2}\}$ is the set of orbits of
%$P_n$.  %$ESS(P_{n})= $ and $ESS_{k}(P_{n})=Ø$ if k=2; (iii)
%
%%We notice that $ens(P_{2k})=(0, 0, 0, \binom{k}{2}, 0,...)$ if
%%$n=2k, \, k\geq 6$ and $ens(P_{2k+1})=(0, 0, k, \binom{k}{2},
%%0,...)$ if $n=2k+1, \, k\geq 5$.
%\end{rem}

%\begin{rem}
%For a path graph of order $n \geq 5$, essential set C
% is exactly union of two orbits of $P_{n}$.
%\end{rem}

Hence, we have the following straightforward proposition for path
graphs $P_{n}$.

\begin{prop}
If $\mathcal{G}$ is a path graph then we have the following: \\
$(i)$ ${ESS}(P_{n})=\emptyset$ for $n\leq 4.$ \\
$(ii)$ $|{ESS_{3}}(P_{2k+1})|= k; k\geq 2$ and
$|{ESS_{3}}(P_{2k})|= 0; k\geq 2.$ \\
$(iii)$ $|{ESS_{4}}(P_{2k})|= \binom{k}{2}; n\geq 4$ \\
$(iv)$ $ens(P_{2k+1})=(0, 0, k, \binom{k}{2}, 0,...)$ if $k\geq 2$
and $ens(P_{2k})=(0, 0, 0, \binom{k}{2}, 0,...)$ if $ k\geq 3$.
%\\$(v)$ For $n\geq 5$, $\mathsf{E}dim(P_n)=\left\{\begin{array}{lll}
% 3 & \mbox{if} &
% $n is odd$\\
% 4 & \mbox{if} & $n is even$
% \end{array}\right.$.
\end{prop}

In the next proposition, we prove that essential set $\mathcal{S}$
for any graph $\mathcal{G}$ is exactly union of two orbits.

\begin{prop}\label{20}
$\mathcal{I}-essential$ set $\mathcal{S}$ is always union of two
orbits of graph $\mathcal{G}$.
\end{prop}

\begin{proof}
Assume contrary that set $\mathcal{S}\subseteq \mathcal{A}$ is a
union of more than two orbits of $\mathcal{G}$ then there exists
$\mathcal{Q}\subset \mathcal{S}$ such that
$\gamma_{\mathcal{A}\setminus \mathcal{Q}}(\mathcal{G})\neq
\gamma_{\mathcal{A}}(\mathcal{G})$, i.e., there exists $x, x\prime
\in \mathcal{A}\setminus \mathcal{Q}$ such that
$x\equiv_{\mathcal{A}\setminus \mathcal{Q}} x\prime $ and $\urcorner
(x\equiv_{\mathcal{A}} x\prime )$ which does not satisfy the
minimality condition for $\mathcal{S}$ to be essential. Hence,
$\mathcal{S}$ is a union of two orbits of a graph $\mathcal{G}$.
\end{proof}

\subsection{Discernibility Matrix of Graph}
In rough set theory, discernibility matrix is a tool to study
information system. Here, we concentrate on the structural concept
of discernibility matrix in the context of graph theory.

Let $\mathcal{I}=\langle \mathcal{V}, \mathcal{V}, \mathcal{V}al,
\mathcal{F}\rangle$ is an information system with
$\mathcal{V}=\{x_1,x_2,,\dots,x_n\}$. The
discernibility matrix $\Delta[\mathcal{I}]$ of $\mathcal{I}$ is an
$n\times n$ matrix with $ijth$ entry, $\Delta_{\mathcal{G}}(x_i,x_j)$, of the matrix is the attribute subset corresponding to the pair
$(x_i,y_j)$ given as:
\begin{center}
$\Delta_{\mathcal{G}}(x_i,x_j)=\{a\in \mathcal{A}:
\mathcal{F}(x_i,a)\neq \mathcal{F}(x_j,a)\}$.
\end{center}

Note that in context of granular referencing system, $\mathcal{A}$
serves as reference set and helps in identification whether given
two vertices are  identical. Two elements in an orbit are
similar for any choice of  $\mathcal{A}$.
%behave alike or otherwise. Note that two elements
%in an orbit behave alike for any choice of  $\mathcal{A}$.
Therefore, two vertices in different orbits are discernible by any
of the vertices in those orbits. We define the entries of the discernibility matrix in the following way:
%\\ For a graph $\mathcal{G}$, if $k_{x}, k_{y} \in \mathcal{V}$,

%\begin{center}
%$\Delta_{\mathcal{G}}(k_{x}, k_{y})=\left\{\begin{array}{lll}
%\mathcal{O}(k_{x})\cup \mathcal{O}(k_{y}) & \mbox{if} & k_{x}\notin \mathcal{O}(k_{y})\\
%\emptyset & \mbox{if} & k_{x}\in\mathcal{O}(k_{y})
%\end{array}\right.$
%\end{center}

For a graph $\mathcal{G}$, if $x_i,x_j \in \mathcal{V}$, then
\begin{center}
$\Delta_{\mathcal{G}}(x_i,x_j)=\left\{\begin{array}{lll}
\mathcal{O}(x_i)\cup\mathcal{O}(x_j) & \mbox{if} & x_i\notin \mathcal{O}(x_j)\\
\emptyset & \mbox{if} & x_i\in\mathcal{O}(x_j)
\end{array}\right.$
\end{center}

In the next theorem, we show that the discernibility matrices of graphs
always characterize graphs uniquely. %\textbf{Not necessarily true.}

\begin{thm}\label{18}
For two graphs $\mathcal{G}_1$ and $\mathcal{G}_2$ such that
$\mathcal{V}(\mathcal{G})=\mathcal{V}(\mathcal{G}_1)=\mathcal{V}(\mathcal{G}_2)$.
$\Delta[\mathcal{G}_1]=\Delta[\mathcal{G}_2]$ if and only if
$\mathcal{G}_1\cong\mathcal{G}_2$.
\end{thm}
\begin{proof}
For any two graphs $\mathcal{G}_1$ and $\mathcal{G}_2$, if
$\mathcal{G}_1\cong\mathcal{G}_2$ then without loss of generality,
$\Delta[\mathcal{G}_1]=\Delta[\mathcal{G}_2]$ (which is a trivial
case). Conversely, suppose that
$\Delta_{\mathcal{G}_1}(x_i,x_j)=\Delta_{\mathcal{G}_2}(x_i,x_j)$
for $x_i,x_j\in\mathcal{V}(\mathcal{G})$, then by definition, it
follows that $\mathcal{F}_{\mathcal{G}_1}(x_i,x_j)=0$ if $x_i \notin
\mathcal{O}(x_j)$ and $\mathcal{F}_{\mathcal{G}_2}(x_i,x_j)=1$ if
$x_i \in \mathcal{O}(x_j)$. We claim that if
$\mathcal{F}_{\mathcal{G}_1}(x_i,x_j)=1=\mathcal{F}_{\mathcal{G}_2}(x_i,x_j)$,
then $x_i \in \mathcal{O}(x_j)$ and $x_j \in \mathcal{O}(x_i)$ which yields that
$\mathcal{O}(x_i)=\mathcal{O}(x_j)$. Thus, preserving the adjacency
and degrees of vertices resulting into
$\mathcal{E}(\mathcal{G}_1)=\mathcal{E}(\mathcal{G}_2)$. Hence
$\mathcal{G}_1\cong\mathcal{G}_2$.

%If $\Delta_{\mathcal{G}_1}(k_x,k_y)=\Delta_{\mathcal{G}_2}(k_x,k_y)$
%for $k_x,k_y\in\mathcal{V}(\mathcal{G})$, then
%$\mathcal{F}_{\mathcal{G}_1}(k_1,k_2)=\mathcal{F}_{\mathcal{G}_2}(k_1,k_2)$
%for each combination of vertices $k_1,k_2$. Since graphs are loop
%free, therefore it follows straightforwardly that
%$\mathcal{E}(\mathcal{G}_1)=\mathcal{E}(\mathcal{G}_2)$. Hence
%$\mathcal{G}_1\cong\mathcal{G}_2$.
\end{proof}

Please note that rows and columns of discernibility matrix
corresponding to similar vertices are equal. Hence instead of
considering individual elements, their orbits can be considered.
Therefore, we first evaluate the classes of orbits of graphs and
then consider the rows and columns of discernibility matrix in terms
of the orbits of graphs. We name the discernibility matrix obtained by considering orbits instead of vertices as Quotient Discernibility Matrix (QDM). QDM has order equal to the number of orbits of the graphs. Rows and columns of QDM are labeled by orbits of the graph and entries of QDM are defined as follows:
For a graph $\mathcal{G}$, with two arbitrary orbits $\mathcal{O}_i$ and $\mathcal{O}_j$
\begin{center}
$\Delta_{\mathcal{G}}(\mathcal{O}_i,\mathcal{O}_j)=\left\{\begin{array}{lll}
\mathcal{O}_i\cup\mathcal{O}_j & \mbox{if} & i \neq j\\
\emptyset & \mbox{if} & i=j.
\end{array}\right.$
\end{center}

%\begin{exam}
%Consider the graph $\mathcal{G}$ shown in Figure 1 and its
%information table shown in Table 1, discernible matrix of
%$\mathcal{G}$ is given in Table 3:
%\begin{table}[h]
%\centering
%\begin{tabular}{|c|c|c|c|} \hline
%  & $\mathcal{O}(0)$ & $\mathcal{O}(2)$ & $\mathcal{O}(3)$ \\
%  \hline
%  $\mathcal{O}(0)$  & $\emptyset$  & $\ast$ & $\ast$ \\
%  $\mathcal{O}(2)$  & $\mathcal{O}(0)\cup\mathcal{O}(2)$ & $\emptyset$ & $\ast$ \\
%  $\mathcal{O}(3)$  & $\mathcal{O}(0)\cup\mathcal{O}(3)$ & $\mathcal{O}(2)\cup\mathcal{O}(3)$ & $\emptyset$\\
%  \hline
%\end{tabular}
%\vspace*{5mm} \caption{Discernibility Matrix $\Delta[\mathcal{I}]$}
%\end{table}
%\FloatBarrier
%\end{exam}

Next, we represent by $EQDM(\mathcal{I}(\mathcal{G}))$, the collection of all distinct entries of $\Delta[\mathcal{I}]$.

In Proposition \ref{19}, we describe a relationship between $EQDM(\mathcal{I}(\mathcal{G}))$ and $ESS(\mathcal{I}(\mathcal{G}))$, the family of $\mathcal{I}$-essential
subsets.

\begin{prop}\label{19}
Let $\mathcal{G}$ be a graph then $EQDM(\mathcal{I}(\mathcal{G}))=ESS(\mathcal{I}(\mathcal{G}))$. \end{prop}
\begin{proof}
Let $S \in ESS(\mathcal{I}(\mathcal{G}))$ then by Proposition \ref{20},
 $S=\mathcal{O}_i \cup \mathcal{O}_j$. Note that $\Delta_{\mathcal{G}}(\mathcal{O}_i,\mathcal{O}_j)
 =\mathcal{O}_i\cup\mathcal{O}_j$ hence $S \subseteq EQDM(\mathcal{I(G)})$ and
 $ESS(\mathcal{I}(\mathcal{G}))\subseteq EQDM(\mathcal{I}(\mathcal{G}))$.
 Conversely, suppose that $S \in EQDM (\mathcal{I(G)})$ then there exist $\mathcal{O}_i$
 and $\mathcal{O}_j$ such that $\Delta_{\mathcal{G}}(\mathcal{O}_i,\mathcal{O}_j)=\mathcal{O}_i\cup\mathcal{O}_j$.
Note that $\gamma_{\mathcal{V}\setminus \{\mathcal{O}_i \cup \mathcal{O}_j\}}\neq \gamma_{\mathcal{V}}$
which satisfies first condition for $S$ to be an essential set.
Let $T\subset \mathcal{O}_i \cup \mathcal{O}_j$ then we have three possibilities. (i).
$T \subseteq \mathcal{O}_i$, (ii). $T \subseteq \mathcal{O}_j$, (iii). $T \cap \mathcal{O}_i\neq \emptyset$
and $T \cap \mathcal{O}_j\neq \emptyset$. It is easy to see in all these three cases that $\gamma_{\mathcal{V}\setminus T }(\mathcal{
G})= \gamma_{\mathcal{V}}(\mathcal{G})$ which satisfies the second condition for $S$ to be an essential set. Hence $S \in EES(\mathcal{I(G)})$ and $EQDM(\mathcal{I}(\mathcal{G}))\subseteq ESS(\mathcal{I}(\mathcal{G}))$ which gives required result.

\end{proof}

\begin{exam}
Consider the graph $\mathcal{G}$ shown in FIGURE 1 and its discernibility matrix is given in the following TABLE 1:
\begin{table}[h]
\centering
\begin{tabular}{|c|c|c|c|} \hline
  & $\mathcal{O}(0)$ & $\mathcal{O}(2)$ & $\mathcal{O}(3)$ \\
  \hline
  $\mathcal{O}(0)$  & $\emptyset$  & $\ast$ & $\ast$ \\
  $\mathcal{O}(2)$  & $\mathcal{O}(0)\cup\mathcal{O}(2)$ & $\emptyset$ & $\ast$ \\
  $\mathcal{O}(3)$  & $\mathcal{O}(0)\cup\mathcal{O}(3)$ & $\mathcal{O}(2)\cup\mathcal{O}(3)$ & $\emptyset$\\
  \hline
\end{tabular}
\vspace*{5mm} \caption{Discernibility Matrix $\Delta[\mathcal{I}]$}
\end{table}
\FloatBarrier
\begin{center}
${ESS}(\mathcal{I}(\mathcal{G}))=
EQDM(\mathcal{I}(\mathcal{G}))=\{\mathcal{O}(1)\mathcal{O}(2),
\mathcal{O}(1)\mathcal{O}(3), \mathcal{O}(2)\mathcal{O}(3) \}$
\end{center}
\FloatBarrier
\end{exam}

\section*{Conclusions}
We have used orbits of graphs to construct information systems to study
graphs using rough set theory. We have studied indiscernibility partition, lower and upper approximations of subsets of vertices of graphs, the rough membership function and rough positive region for
some well-known families of graphs like cycle, complete and complete
bipartite graphs. We have also studied the essential sets of graphs and the
discernibility matrix of graphs in terms of orbit partitions. Identifying vertices having similar characteristics simplifies the structures of graphs. The terminology emerging from the merger of rough set theory and graph theory will be useful for exploring new problems associated to symmetries of graphs.

% ------------------------------------------------------------------------
%Included for Gather Purpose only:
%input "Xbib.bib"
%\bibliographystyle{amsplain}
%\bibliography{xbib}

\begin{thebibliography}{999}


\bibitem{2}
Cattaneo, G. (2011). An investigation about rough set theory: Some foundational and mathematical aspects. Fundamenta Informaticae, 108(3-4), 197-221.
\bibitem{3}
Cattaneo, G., Chiaselotti, G., Ciucci, D., and Gentile, T. (2016). On the connection of hypergraph theory with formal concept analysis and rough set theory. Information Sciences, 330, 342-357.
\bibitem{4}
Chen, J., and Li, J. (2012). An application of rough sets to graph theory. Information Sciences, 201, 114-127.
\bibitem{5}
Chiaselotti, G., Ciucci, D., and Gentile, T. (2015). Simple undirected graphs as formal contexts. Formal Concept Analysis, 9113, 287-302.
\bibitem{6}
Chiaselotti, G., Ciucci, D., Gentile, T., and Infusino, F. (2015).
Rough set theory applied to simple undirected graphs. Rough Sets and
Knowledge Technology, 423-434.
\bibitem{7}
Chiaselotti, G., Ciucci, D., and Gentile, T. (2016). Simple graphs in granular computing. Information Sciences, 340, 279-304.


\bibitem{9}
Hońko, P. (2012). Relational pattern updating. Information Sciences, 189, 208-218.
\bibitem{10}
Hońko, P. (2013). Association discovery from relational data via granular computing. Information Sciences, 234, 136-149.
\bibitem{11}
Huang, B., Zhuang, Y., Li, H., and Wei, D. (2013). A dominance intuitionistic fuzzy-rough set approach and its applications. Applied Mathematical Modelling, 37(12-13), 7128-7141.
\bibitem{12}
Kang, X., Li, D., Wang, S., amd Qu, K. (2012). Formal concept analysis based on fuzzy granularity base for different granulations. Fuzzy Sets and Systems, 203, 33-48.
\bibitem{13}
Lin, T. Y. (1999). Data Mining: Granular computing approach. Methodologies for Knowledge Discovery and Data Mining, 1574, 24-33.
\bibitem{14}
Lin, T. Y. (2000). Data mining and machine oriented modeling: a granular approach. Applied Intelligence, 13(2), 113-124.
\bibitem{15}
Midelfart, H., and Komorowski, J. (2002). A rough set framework for learning in a directed acyclic graph. Rough Sets and Current Trends in Computing, 2475, 144-155.
\bibitem{16}
Pawlak, Z. (1991). Rough sets: Theoretical aspects of reasoning about data. Kluwer Academic Publishers.
\bibitem{17}
Pawlak, Z., and Skowron, A. (2007a). Rough sets and boolean reasoning. Information Sciences, 177(1), 41-73.
\bibitem{18}
Pawlak, Z., and Skowron, A. (2007b). Rough sets: Some extensions. Information Sciences, 177(1), 28-40.
\bibitem{19}
Pawlak, Z., and Skowron, A. (2007c). Rudiments of rough sets. Information Sciences, 177(1), 3-27.
\bibitem{20}
Pedrycz, A., Hirota, K., Pedrycz, W., and Dong, F. (2012). Granular representation and granular computing with fuzzy sets. Fuzzy Sets and Systems, 203, 17-32.
\bibitem{21}
Qiu, T., Chen, X., Liu, Q., and Huang, H. (2010). Granular computing approach to finding association rules in relational database. International Journal of Intelligent Systems, 25, 165-179.
\bibitem{22}
Skowron, A., and Rauszer, C. (1992). The discernibility matrices and functions in information systems. Intelligent Decision Support, 11, 331-362.
\bibitem{23}
Tang, J., She, K., and Zhu, W. (2012). Matroidal structure of rough sets from the viewpoint of graph theory. Journal of Applied Mathematics, 2012, 1-27.
\bibitem{24}
Wang, C., and Chen, D. (2010). A short note on some properties of rough groups. Computers and Mathematics with Applications, 59(1), 431-436.
\bibitem{25}
Wang, S., Zhu, Q., Zhu, W., and Min, F. (2013). Equivalent characterizations of some graph problems by covering-based rough sets. Journal of Applied Mathematics, 2013, 1-7.
\bibitem{26}
Wei-Zhi Wu, Yee Leung, and Ju-Sheng Mi. (2009). Granular computing and knowledge reduction in formal contexts. IEEE Transactions on Knowledge and Data Engineering, 21(10), 1461-1474.
\bibitem{27}
Yao, Y. Y. (2001). On modeling data mining with granular computing. 25th Annual International Computer Software and Applications Conference. COMPSAC 2001, 2001, 638-643.
\bibitem{28}
Yao, Y. Y. (2002). A granular computing approach to machine learning. Proceedings of the 1st International Conference on Fuzzy Systems and Knowledge Discovery, 732-736.
\bibitem{29}
Zadeh, L. A. (1996). Fuzzy set and information granularity. Advances in Fuzzy Systems — Applications and Theory, 433-448.

\end{thebibliography}
\end{document}